\documentclass[12pt,twoside]{article}
\usepackage{amsmath,amssymb,amsthm,mathrsfs,color,times,textcomp,sectsty,verbatim}
\usepackage[colorinlistoftodos]{todonotes}
\usepackage[colorlinks=true]{hyperref}
\usepackage{xcolor}
\hypersetup{urlcolor=blue, citecolor=red, linkcolor=blue}
 \def \no{\nonumber}

 \allowdisplaybreaks

\def\q{\mathcal{Q}_{a,b}^{f,h}}
\def\im{\int_M}

\def\dmz{d\mu_{g_0}}
\def\idm{\int_{\partial M}}
\def\ds{d\sigma_g}
\def\dsz{d\sigma_{g_0}}

\def\pa{\partial}
\def\e{\epsilon}

\def\criti{\frac{2n}{n-2}}
\def\critb{\frac{2(n-1)}{n-2}}

\begin{document}
\renewcommand{\proofname}{\bf Proof}
\let\oldsection\section
\renewcommand\section{\setcounter{equation}{0}\oldsection}
\renewcommand\thesection{\arabic{section}}
\renewcommand\theequation{\thesection.\arabic{equation}}
\newtheorem{claim}{\indent Claim}[section]
\newtheorem{theorem}{\indent Theorem}[section]
\newtheorem{lemma}{\indent Lemma}[section]
\newtheorem{proposition}{\indent Proposition}[section]
\newtheorem{definition}{\indent Definition}[section]
\newtheorem{remark}{\indent Remark}[section]
\newtheorem{corollary}{\indent Corollary}[section]
\newtheorem{example}{\indent Example}[section]
\title{\Large \bf
Prescribed scalar curvature plus mean curvature flows in compact manifolds with boundary of negative conformal invariant}

\author{Xuezhang  Chen\thanks{X. Chen: xuezhangchen@nju.edu.cn; $^\dag$P. T. Ho: ptho@sogang.ac.kr/paktungho@yahoo.com.hk; $^\ddag$L. Sun: ls680@math.rutgers.edu.}~, Pak Tung Ho$^\dag$ and Liming Sun$^\ddag$\\
 \small
$^\ast$Department of Mathematics, Nanjing University, Nanjing
210093, P. R. China\\
\small
$^\dag$Department of Mathematics, Sogang University, Seoul 121-742, Korea\\
\small
$^\ddag$Department of Mathematics, Rutgers University, 110 Frenlinghuysen Road, Piscataway NJ 08854, USA
}

\date{}
\maketitle
\begin{abstract}
We employ three different methods to prove the following result on prescribed scalar curvature plus mean curvature problem: Let $(M^n,g_0)$ be a
$n$-dimensional smooth compact manifold with boundary, where $n \geq 3$, assume the conformal invariant $Y(M,\partial M)<0$. Given any negative smooth functions $f$ in $M$ and $h$ on $\partial M$, there exists a unique conformal metric of $g_0$ such that its scalar curvature equals $f$ and mean curvature curvature equals $h$. The first two methods are sub-super-solution method and subcritical approximation, and the third method is a geometric flow. In the flow approach, assume another conformal invariant $Q(M,\pa M)$ is a negative real number, for some class of initial data, we prove the short time and long time existences of the so-called prescribed scalar curvature plus mean curvature flows, as well as their asymptotic convergence. Via a family of such flows together with some additional variational arguments, under the flow assumptions we prove existence and uniqueness of positive minimizers of the associated energy functional and also the above result by analyzing asymptotic limits of the flows and the relations among some conformal invariants.

{{\bf $\mathbf{2010}$ MSC:} Primary 53C44, 53C21, 35B50; Secondary 35J60, 35K55, 35R01.}

\end{abstract}

\section{Introduction}

Let $(M,g_0)$ be a $n$-dimensional smooth compact Riemannian manifold without boundary,
where $n\geq 3$. As a generalization of Uniformization Theorem,
the Yamabe problem is to find a metric conformal to $g_0$ such that its scalar curvature is constant.
This problem was solved by Yamabe, Trudinger, Aubin and Schoen. See the survey article \cite{Lee&Parker} by Lee and Parker
for more details.
The Yamabe flow is a geometric flow introduced to tackle the Yamabe problem.
See \cite{Brendle4,Brendle5} and references therein
for results of Yamabe flow on closed manifolds.

Analogues of the Yamabe problem have been studied in smooth compact Riemannian manifolds with boundary:
Find a conformal metric such that its scalar curvature
in $M$ equals $c_1$ and its mean curvature on the boundary $\partial M$ equals $c_2$,
where $c_1$ and $c_2$ are two constants.
The special case when $c_1=0$ or $c_2=0$ was studied by Escobar \cite{escobar1,escobar3} (see also \cite{Marques1, Marques2, Brendle&Chen, Almaraz,Almaraz1}),
and the corresponding
geometric flow was introduced and studied by Brendle \cite{Brendle1} (see also \cite{Almaraz,Almaraz-Sun}).
The general case was studied by Araujo \cite{Araujo1, Araujo2}, Cherrier \cite{cherrier}, Escobar \cite{escobar}  and Han-Li \cite{han-li1,han-li2}.
Among them, we mention the following result due to Han-Li:
Let $(M,g_0)$ be a $n$-dimensional smooth compact Riemannian manifold with boundary which has positive
(generalized) Yamabe constant (see (\ref{Yamabe_constant}) below for the definition). Assume that the boundary has at least one nonumbilic point with $n \geq 5$, or $(M,g_0)$ is locally conformally flat with umbilic boundary for $n\geq 3$. Then there exists a metric conformal to $g_0$ such that its scalar curvature in $M$ is equal to $1$
and its mean curvature on $\partial M$ is equal to any given constant $c$. Araujo \cite{Araujo1, Araujo2} gave some characterizations of critical points (including the minimizers) for the total curvature plus total mean curvature functional (see \eqref{eq:energy}) with some volume constraints, as well as its second variation.

More generally, we would like to study the following prescribed scalar curvature plus mean curvature problem:
For a $n$-dimensional smooth compact Riemannian manifold $(M,g_0)$ with boundary $\partial M$,
can we find a conformal metric such that
its scalar curvature equals a given function $f$ in $M$ and
its mean curvature equals a given function $h$ on $\partial M$? This problem is equivalent to
finding a positive solution $u$
 of the following PDE problem:
\begin{align}\label{problem:smc}
\left\{\begin{array}{ll}
\displaystyle-\frac{4(n-1)}{n-2}\Delta_{g_0}u+R_{g_0}u=fu^{\frac{n+2}{n-2}},&\hspace{2mm}\mbox{ in }M,\\
\displaystyle\frac{\partial u}{\partial \nu_0}+\frac{n-2}{2}h_{g_0}u=\frac{n-2}{2}h u^{\frac{n}{n-2}},&\hspace{2mm}\mbox{ on }\partial M,
\end{array}
\right.
\end{align}
where $R_{g_0}$ is scalar curvature, $h_{g_0}$ is mean curvature and $\nu_0=\nu_{g_0}$ is the unit outward normal on $\partial M$. This problem has been studied by Zhang \cite{ZhangL} in dimension three. More background can be found in \cite{ZhangL,escobar4, Djadli-Mal-Ahm} and the references therein.

Before stating our results, we need to introduce some notations. The (generalized) Yamabe constant $Y(M,\partial M)$ is defined as
\begin{equation}\label{Yamabe_constant}
Y(M,\partial M):=\inf_{g\in[g_0]}\frac{\int_MR_gd\mu_g+2(n-1)\int_{\partial M}h_gd\sigma_g}{(\int_{M}d\mu_g)^{\frac{n-2}{n}}}.
\end{equation}
Similarly, we can define (cf. \cite{escobar1})
\begin{equation*}
Q(M,\partial M):=\inf_{g\in[g_0]}\frac{\int_MR_gd\mu_g+2(n-1)\int_{\partial M}h_gd\sigma_g}{(\int_{M}\ds)^{\frac{n-2}{n-1}}}.
\end{equation*}
It was first pointed out by Zhiren Jin (cf. \cite{escobar2}) that $Q(M,\partial M)$ could be $-\infty$, meanwhile $Y(M,\partial M)>-\infty$.  Moreover, it is easy to show that if $Q(M,\partial M)$ is negative and finite, then $Y(M,\pa M)<0$. The total scalar curvature plus total mean curvature functional is defined as
\begin{equation}\label{eq:energy}
E[u]=\int_M (\tfrac{4(n-1)}{n-2}|\nabla u|_{g_0}^2 +R_{g_0}u^2) d\mu_{g_0}+2(n-1)\int_{\partial M}h_{g_0}u^2 d\sigma_{g_0}.
\end{equation}
We define
\begin{equation*}
\q[u]=\frac{E[u]}{a\Big(\int_M -f|u|^{\frac{2n}{n-2}}d\mu_{g_0}\Big)^{\frac{n-2}{n}}+2(n-1)b \Big(\int_{\partial M}-h|u|^{\frac{2(n-1)}{n-2}}d\sigma_{g_0}\Big)^{\frac{n-2}{n-1}}}.
\end{equation*}
Compared with the constraint functional in \cite{escobar},
$\q[u]$ has the advantage that it is homogeneous in $u$.  A direct computation shows that up to some constant multiples of $f$ and $h$, every critical point of $\q$ satisfies problem \eqref{problem:smc}.

The following theorem gives a partial answer to the above problem under the assumption that $Y(M,\partial M)<0$.
The readers may compare our theorem with \cite[Theorem 0.3]{han-li2}.
\begin{theorem}\label{Thm_sub_super_sols}
Let $(M,g_0)$ be a $n$-dimensional smooth compact Riemannian manifold with boundary $\partial M$, where $n \geq 3$. If $Y(M,\pa M)<0$, then for any negative smooth functions $f$ in $M$  and $h$ on $\partial M$ respectively, there exists a unique conformal metric $g \in [g_0]$ such that $R_g=f$ and $h_g=h$.
\end{theorem}

Theorem \ref{Thm_sub_super_sols} is proved in Section \ref{Sect2},
and we outline some ideas to show that there is a unique positive solution to PDE problem \eqref{problem:smc}. By the assumption that $Y(M,\partial M)<0$, we can find a background metric $g_0$ with $R_{g_0}<0$ and $h_{g_0}<0$. Selecting a suitable constant $\e>0$, one can show that $\frac{1}{\e}$ and $\e$ are respectively the super and sub solutions to the problem \eqref{problem:smc}. By the standard elliptic theory, there exists a solution $u$ satisfying $\e \leq u \leq \frac{1}{\e}$. The uniqueness of the problem \eqref{problem:smc} was proved by Escobar \cite[Corollary 2]{escobar4}.

For the last few decades, geometric flows have played an important role in prescribed curvature problems in  conformal geometry. Therefore, a different approach is to introduce some negative gradient flow of the energy $\q$, which we call prescribed scalar curvature plus mean curvature flows with interplay between scalar curvature  in $M$ and mean curvature on the boundary $\partial M$. Indeed, depending on the sign of $Y(M,\pa M)$, we set up a series of projects to solve problem \eqref{problem:smc} by using geometric flows. As a preliminary step, we first consider the negative generalized Yamabe constant. Given two constants $a, b>0$, let $f$ be a negative smooth function in $M$ and $h$ be a  negative smooth function on $\partial M$,  the flow metric $g(x,t)=u(x,t)^{\frac{4}{n-2}}g_0$ satisfies
\begin{equation}\label{mixed_cur_flows}
\begin{split}
\left\{\begin{array}{ll}
\displaystyle\partial_t g=(\alpha(t) f^{-1}R_g-\lambda(t))g &\hspace{2mm}\mbox{ in }M,\\
\displaystyle\partial_t g=(\beta(t)h^{-1}h_g-\lambda(t))g&\hspace{2mm}\mbox{ on }\partial M,\\
u(0)=u_0 \in C^\infty(\bar M),
\end{array}
\right.
\end{split}
\end{equation}
where
\begin{equation}\label{eqn:constrant_time_funcs}
\begin{split}
&\alpha(t)=\frac{1}{a}\Big(\int_M -fd\mu_g\Big)^{\frac{2}{n}}, \quad \beta(t)=\frac{1}{b}\Big(\int_{\partial M}-hd\sigma_g\Big)^{\frac{1}{n-1}},\\
&\lambda(t)=-\frac{E[u(t)]}{a\Big(\int_M -f d\mu_g\Big)^{\frac{n-2}{n}}+2(n-1)b\Big(\int_{\partial M}-h d\sigma_g\Big)^{\frac{n-2}{n-1}}}.
\end{split}
\end{equation}
Here $d\mu_g=u^{\frac{2n}{n-2}}d\mu_{g_0}$ and $d\sigma_g=u^{\frac{2(n-1)}{n-2}}d\sigma_{g_0}$, $R_g$ and $h_g$ are the scalar curvature and the mean curvature of the flow metric $g$, respectively. The selections of $\alpha(t),\beta(t),\lambda(t)$ are due to two reasons: one is to make the energy $\q$ non-increasing along the flows, the other is to preserve the following quantity:
$$a\Big(\int_M -f d\mu_g\Big)^{\frac{n-2}{n}}+2(n-1)b\Big(\int_{\partial M}-h d\sigma_g\Big)^{\frac{n-2}{n-1}}.$$

Our flows are somehow inspired by the one defined on surfaces with boundary, which was first studied by Brendle \cite{Brendle2}. We obtain the following result for geometric flows (\ref{mixed_cur_flows}).
\begin{theorem}\label{mainthm}
Let $(M,g_0)$ be a $n$-dimensional smooth compact Riemannian manifold with boundary $\partial M$, where $n \geq 3$. Assume $Q(M,\pa M)$ is a negative real number. Let $f$ be any negative smooth function in $M$  and $h$ be any negative smooth function on $\partial M$. Given any $a,b>0$ and some class of initial data, the flow metric \eqref{mixed_cur_flows} exists for all time and converges to a smooth conformal metric $g_\infty$ of $g_0$, such that
$$R_{g_\infty}=\frac{\lambda_\infty}{\alpha_\infty} f \hbox{~~and~~} h_{g_\infty}=\frac{\lambda_\infty}{\beta_\infty}h,$$
where $\lambda_\infty=\lim_{t \to \infty}\lambda(t), \alpha_\infty=\lim_{t \to \infty}\alpha(t)$ and $\beta_\infty=\lim_{t \to \infty}\beta(t)$ are positive constants.
\end{theorem}
The above theorem extends the first and second authors' results in \cite{ChenHo}. We point out that it follows from
Theorem \ref{mainthm} that there exists a conformal metric of $g_0$ such that its scalar curvature equals $f$ and its mean curvature curvature equals $c_\infty h$ for some positive constant $c_\infty=c_\infty(a,b, f,h)$. In general, $c_\infty$ is not necessarily equal to 1.

Here we sketch the proof of Theorem \ref{mainthm}, which covers Sections \ref{Sect3}-\ref{Sect6}.  Along the flow (\ref{mixed_cur_flows}),
 $\q[u(t)]$ is non-increasing. Together with the above conserved quantity, we have the uniform bounds for $\alpha(t), \beta(t)$ and $\lambda(t)$
on any finite time interval. In particular, $\lambda(t)$ is nondecreasing. For a special class of initial data, we employ maximum principle to derive the uniform bounds for conformal factor $u(x,t)$, and uniform lower bounds for both scalar curvature $R_{g(t)}$ and mean curvature $h_{g(t)}$. Based on these, we eventually obtain $L^p(M,g_0)$ estimates with all $p \geq 2$ for $R_{g(t)}$ and $h_{g(t)}$ on any finite time interval $[0,T]$. We should mention that the negativity of $Y(M,\pa M)$ plays an important role in the above estimates.  By the standard parabolic theory, the long time existence of the flows follows. With the help of Simon-Lojasiewicz inequality, we establish the asymptotic convergence of the flows and arrive at an important integral estimate, which enables us to conclude that the asymptotic limits of $\alpha(t)$ and $\beta(t)$ exist. Moreover, we give a criterion of the uniqueness of asymptotic limits for different qualified initial data. Roughly speaking, if the energies of asymptotic limits (equivalently critical values of $\q$) for different initial data  coincide, so do asymptotic limits themselves.

As applications of these flows, in Section \ref{Sect7} we prove existence and uniqueness of minimizers for $\q$ by constructing a minimizing sequence of asymptotic limits, which are generated by a $(a,b)$-family of such flow equations \eqref{mixed_cur_flows}. Furthermore, through some delicate analysis on the relations among some conformal invariants, we establish that there exist some $a,b>0$ and the corresponding minimizer $u_{a,b}$ (up to a positive constant scaling) for $\q$ such that $R_{g_{a,b}}=f$ and $h_{g_{a,b}}=h$, where $g_{a,b}=u_{a,b}^{4/(n-2)}g_0$. This means that the unique solution in Theorem \ref{Thm_sub_super_sols} can be realized by the unique minimizer for $\q$ with some suitable $a,b>0$. However, the uniqueness of minimizers for $\q$, whose proof is included in Proposition \ref{prop:criterion_asy_limits},  can not directly follow from the uniqueness of solutions to  PDE problem \eqref{problem:smc}.   We also mention that the minimizer for $\q$ can also be obtained by subcritical approximation (see Remark \ref{rem:subcritical_approx}).

Compared with the aforementioned sub-super-solution method, our flow approach
has the main advantage that it can also be adapted to any compact manifold of positive conformal invariant, which will be left to our ongoing work. So the techniques and estimates developed here will benefit us a lot in the future.

\noindent {\bf Acknowledgments:} This work was carried out during the first author's visit at Rutgers University. He is grateful to Professor Yanyan Li for the invitation to Rutgers University and fruitful discussions.  He also would like to thank  mathematics department at Rutgers University for its hospitality and financial support. He is supported by NSFC (No.11201223), A Foundation for the Author of National Excellent Doctoral Dissertation of China (No.201417), Program for New Century Excellent Talents in University (NCET-13-0271), the travel grants from AMS Fan fund and Hwa Ying foundation at Nanjing University. The second author is supported by the National Research Foundation of Korea (NRF) grant (No.201531021.01) funded by the Korea government (MEST).

\section{Proof of Theorem \ref{Thm_sub_super_sols}}\label{Sect2}

 This section is devoted to the proof of Theorem \ref{Thm_sub_super_sols} by using
Kazdan-Warner's sub-super-solution method \cite{kazdan-warner} with minor modifications.

\noindent {\bf Proof of Theorem \ref{Thm_sub_super_sols}.~~}
Since $Y(M,\partial M)<0$, it follows from assertion \eqref{initial_metric} below that we may assume the background metric $g_0$ satisfies $R_{g_0}<0$ and $h_{g_0}<0$. If we choose
$$0<\epsilon<\min\left\{\min_M \left(\frac{f}{R_{g_0}}\right)^{\frac{n-2}{4}}, \min_M\left(\frac{h}{h_{g_0}}\right)^{\frac{n-2}{2}},\min_M \left(\frac{R_{g_0}}{f}\right)^{\frac{n-2}{4}}, \min_M\left(\frac{h_{g_0}}{h}\right)^{\frac{n-2}{2}}\right\},$$
then  $\bar u=\frac{1}{\epsilon}$ and $\underline u=\epsilon$ are super and sub solutions to \eqref{problem:smc}, respectively. We claim that there exists a solution $u$ to problem \eqref{problem:smc} satisfying $\underline u\leq u\leq \bar u$.

Fix large constants $N=N(\epsilon)>0$ and $H=H(\epsilon)>0$ and let $F(x,u)=fu^{\frac{n+2}{n-2}}+Nu, G(x,u)=\frac{n-2}{2}(hu^{\frac{n}{n-2}}+Hu)$, there hold
$R_{g_0}+N>0$ in $M$, $h_{g_0}+H>0$ on $\partial M$
and
\begin{equation} \label{monotone:F_G}
\partial_s F(x,s)>0,\partial_s G(x,s)>0~~\forall~~ \underline u\leq s\leq \bar u.
\end{equation}
Rewrite the equations in problem \eqref{problem:smc} as
\begin{align*}
\tilde Lu:=&-\frac{4(n-1)}{n-2}\Delta_{g_0} u+(R_{g_0}+N)u=-fu^{\frac{n+2}{n-2}}+Nu=F(x,u) \hbox{~~in~~} M,\\
\tilde Bu:=&\frac{\partial u}{\partial \nu_0}+\frac{n-2}{2}(h_{g_0}+H)u=\frac{n-2}{2}(hu^{\frac{n}{n-2}}+Hu) \hbox{~~on~~} \partial M.
\end{align*}
Furthermore, we have
\begin{align*}
\tilde L\bar u \geq F(x,\bar u) \hbox{~~in~~} M, ~~\tilde B\bar u \geq G(x,\bar u) \hbox{~~on~~} \partial M;\\
\tilde L\underline u \leq F(x,\underline u) \hbox{~~in~~} M, ~~\tilde B\underline u \leq G(x,\underline u) \hbox{~~on~~} \partial M.
\end{align*}
Let $u_1=\underline u$ and define
\begin{align*}
\tilde Lu_{i+1}=F(x,u_i) \hbox{~~in~~} M, ~~\tilde B u_{i+1}=G(x,u_i) \hbox{~~on~~} \partial M
\end{align*}
for all $i \in \mathbb{N}$. Since $\tilde Lu_2=F(x,u_1)=F(x,\underline u)\geq \tilde L\underline u$ and $\tilde B u_2=G(x,u_1)=G(x,\underline u)\geq \tilde B \underline u$, the maximum principle for $\tilde L$ and $\tilde B$ gives $u_2 \geq u_1=\underline u$. By induction, if $u_i\geq u_{i-1}$ for some $i \geq 2$, then $u_{i+1}\geq u_i$ follows by applying maximum principle to $\tilde Lu_{i+1}=F(x,u_i)\geq F(x,u_{i-1})\geq  \tilde Lu_{i-1}$ and $\tilde B u_{i+1}=G(x,u_i)\geq G(x,u_{i-1})\geq  \tilde B u_{i-1}$, where we have used \eqref{monotone:F_G}. On the other hand, notice that $u_1=\underline u \leq \bar u$. By induction, if $u_i \leq \bar u$, similarly $u_{i+1}\leq \bar u$ follows from $\tilde Lu_{i+1}=F(x,u_i)\leq F(x,\bar u)\leq \tilde L\bar u$ and $\tilde Bu_{i+1}=G(x,u_i)\leq G(x,\bar u)\leq \tilde B\bar u$. Hence we obtain a sequence of nondecreasing functions $\{u_i\}$ satisfying $\forall i, \underline u\leq u_i \leq \bar u, u_i \in C^{2,\alpha}(M),~0<\alpha<1, |F(x,u_i)|\leq C$ and $|G(x,u_i)|\leq C$, where $C$ is a constant independent of $i$. Then from
\begin{align*}
\left\{\begin{array}{ll}
\displaystyle-\frac{4(n-1)}{n-2}\Delta_{g_0} u_{i+1}+(R_{g_0}+N)u_{i+1}=F(x,u_i),&\hspace{2mm}\mbox{ in }M,\\
\displaystyle\frac{\partial u_{i+1}}{\partial \nu_0}+\frac{n-2}{2}(h_{g_0}+H)u_{i+1}=G(x,u_i),&\hspace{2mm}\mbox{ on }\partial M,
\end{array}
\right.
\end{align*}
we get that $\forall~ 1<q<\infty$, there holds $\|u_{i+1}\|_{W^{2,q}(M,g_0)}\leq C(q)$, where $C(q)$ is a constant independent of $i$ (regularity of solutions has been established in \cite{cherrier} or \cite{Araujo2}). By Sobolev embedding theorem, we get $\|u_i\|_{C^{1,\beta}(\bar M)}\leq C(\beta),~~\forall~ 0<\beta<1$. Also Schauder estimate gives $\|u_i\|_{C^{k,\alpha}(\bar M)}\leq C(k),~~\forall~ k \in \mathbb{N}$. From the monotonicity and uniform boundedness of $\{u_i\}$, for any fixed $x \in \bar M$, there exists some function $u(x)$ such that $u_i(x)\to u(x)$ as $i \to \infty$. Moreover, Ascoli-Arzela theorem shows that there exists a subsequence of $\{u_i\}$ converges in $C^{k,\alpha}(\bar M)$ to $u$. Therefore, we conclude that the whole sequence of $\{u_i\}$ converges in $C^{k,\alpha}(\bar M)$ to $u$. This finishes the proof.
\hfill{$\Box$}

\section{Flow equations and some elementary estimates}\label{Sect3}

From now on, we assume that $(M,g_0)$ is a $n$-dimensional smooth compact Riemannian manifold with boundary $\partial M$, where $n \geq 3$ and  $Q(M,\partial M)$ is a negative real number. For brevity, we denote by $L_{g_0}=-\frac{4(n-1)}{n-2}\Delta_{g_0}+R_{g_0}$ the conformal Laplacian and $B_{g_0}=\frac{\partial}{\partial \nu_0}+\frac{n-2}{2}h_{g_0}$ the boundary conformally covariant operator.

In terms of the conformal factor $u$, the prescribed scalar curvature plus mean curvature flows \eqref{mixed_cur_flows} can be written as
\begin{equation}\label{mixed_cur_flows_u}
\begin{split}
\left\{\begin{array}{ll}
\displaystyle\partial_t u=\frac{n-2}{4}(\alpha(t) f^{-1}R_g-\lambda(t))u &\hspace{2mm}\mbox{ in }M,\\
\displaystyle\partial_t u=\frac{n-2}{4}(\beta(t)h^{-1}h_g-\lambda(t))u&\hspace{2mm}\mbox{ on }\partial M,\\
u(0)=u_0 \in C^\infty(\bar M).
\end{array}
\right.
\end{split}
\end{equation}
In addition, we have
\begin{align}
L_{g_0}u=&R_g u^{\frac{n+2}{n-2}}  \hbox{~~in~~} M,\label{eq:sc}\\
B_{g_0}u=&\frac{n-2}{2}h_gu^{\frac{n}{n-2}} \hbox{~~on~~} \partial M,\label{eq:mc}
\end{align}
which implies that
\begin{equation*}
E[u(t)]=\int_M R_g d\mu_g+2(n-1)\int_{\partial M}h_g d\sigma_g.
\end{equation*}
It follows from \eqref{mixed_cur_flows_u}, \eqref{eq:sc} and \eqref{eq:mc} that
$u$ satisfies the following evolution equations:
\begin{align}\label{evolution_u}
\left\{\begin{array}{ll}
\displaystyle\partial_t \big(u^{\frac{n+2}{n-2}}\big)=\frac{n+2}{4}\Big(
-\frac{4(n-1)}{n-2}\alpha(t) f^{-1}\Delta_{g_0}u+\alpha(t) f^{-1}R_{g_0}u
-\lambda(t)u^{\frac{n+2}{n-2}}\Big),&\mbox{ in }M,\\
\displaystyle\partial_t \big(u^{\frac{n}{n-2}}\big)=\frac{n}{4}\Big(\frac{2}{n-2}\beta(t)h^{-1}\frac{\partial}{\partial \nu_0}u+\beta(t)h^{-1}h_{g_0}u-\lambda(t)u^{\frac{n}{n-2}}\Big), &\mbox{ on }\partial M,\\
\displaystyle u>0 \hbox{~~in~~} \overline M \hbox{~~and~~}u(0)=u_0.
\end{array}
\right.
\end{align}
The proof of the short time existence of the flows \eqref{mixed_cur_flows} is deferred to Section \ref{Sec3}.

A direct computation yields
\begin{align*}
&\frac{d}{dt}\Big[a\Big(\int_M -fd\mu_g\Big)^{\frac{n-2}{n}}+2(n-1)b\Big(\int_{\partial M}-hd\sigma_g\Big)^{\frac{n-2}{n-1}}\Big]\\
=&-\frac{n-2}{2}\Big[-a \Big(\int_M -f d\mu_g\Big)^{-\frac{2}{n}}\int_M f(\lambda-\alpha f^{-1}R_g)d\mu_g\\
&\qquad\qquad-2(n-1)b\Big(\int_{\partial M}-h d\sigma_g\Big)^{-\frac{1}{n-1}}\int_{\partial M}h(\lambda-\beta h^{-1}h_g)d\sigma_g\Big]\\
=&-\frac{n-2}{2}\Big\{\lambda\Big[a\Big(\int_M -fd\mu_g\Big)^{\frac{n-2}{n}}+2(n-1)b\Big(\int_{\partial M}-hd\sigma_g\Big)^{\frac{n-2}{n-1}}\Big]\\
&+a\alpha\Big(\int_M -f d\mu_g\Big)^{-\frac{2}{n}}\int_M R_g d\mu_g+2(n-1)b\beta\Big(\int_{\partial M}-h d\sigma_g\Big)^{-\frac{1}{n-1}}\int_{\partial M}h_g d\sigma_g\Big\}\\
=&-\frac{n-2}{2}\Big\{\lambda\Big[a\Big(\int_M -fd\mu_g\Big)^{\frac{n-2}{n}}+2(n-1)b\Big(\int_{\partial M}-hd\sigma_g\Big)^{\frac{n-2}{n-1}}\Big]+E[u(t)]\Big\}\\
=&0.
\end{align*}
From this, we may normalize the initial data such that
\begin{align}\label{conserved_quantity}
&a\Big(\int_M -fu(t)^{\frac{2n}{n-2}}d\mu_{g_0}\Big)^{\frac{n-2}{n}}+2(n-1)b \Big(\int_{\partial M}-hu(t)^{\frac{2(n-1)}{n-2}}d\sigma_{g_0}\Big)^{\frac{n-2}{n-1}}\no\\
=&a\big(\int_M -fu_0^{\frac{2n}{n-2}}d\mu_{g_0}\Big)^{\frac{n-2}{n}}+2(n-1)b \Big(\int_{\partial M}-hu_0^{\frac{2(n-1)}{n-2}}d\sigma_{g_0}\Big)^{\frac{n-2}{n-1}}=1
\end{align}
for all time $t \geq 0$.
We remark that the flow equations \eqref{mixed_cur_flows} can be regarded as a negative gradient flow for the energy functional $\q$ with constraint \eqref{conserved_quantity}, since along such curvature flows \eqref{mixed_cur_flows}, the energy $E[u(t)]$ is non-increasing, so is $\q$. Indeed, from \eqref{eqn:constrant_time_funcs},
 \eqref{eq:sc} and  \eqref{eq:mc}, we obtain
\begin{align}\label{energy_rate}
&\frac{d}{dt}E[u(t)]\no\\
=&2\Big[\frac{4(n-1)}{n-2}\int_M \langle \nabla u,\nabla u_t\rangle_{g_0}d\mu_{g_0}+\int_M R_{g_0}uu_t d\mu_{g_0}+2(n-1)\int_{\partial M} h_{g_0}uu_t d\sigma_{g_0}\Big]\no\\
=& 2\Big[\int_M u_t L_{g_0}(u)d\mu_{g_0}+\frac{4(n-1)}{n-2}\int_{\partial M}B_{g_0}(u) u_t d\sigma_{g_0}\Big]\no\\
=&\frac{n-2}{2}\Big[\int_M R_g(\alpha f^{-1}R_g-\lambda)d\mu_g+2(n-1)\int_{\partial M}h_g(\beta h^{-1}h_g-\lambda)d\sigma_g\Big]\no\\
=&-\frac{n-2}{2}\Big[-\alpha^{-1} \int_Mf(\lambda-\alpha f^{-1}R_g)^2d\mu_g-2(n-1)\beta^{-1}\int_{\partial M}h(\lambda-\beta h^{-1}h_g)^2d\sigma_g\Big]\no\\
&+\frac{n-2}{2}\lambda\Big\{\lambda\Big[\alpha^{-1}\int_M -fd\mu_g+2(n-1) \beta^{-1}\int_{\partial M}-hd\sigma_g\Big]+E[u(t)]\Big\}\no\\
=&-\frac{n-2}{2}\Big[\alpha^{-1} \int_M-f(\lambda-\alpha f^{-1}R_g)^2d\mu_g+2(n-1)\beta^{-1}\int_{\partial M}-h(\lambda-\beta h^{-1}h_g)^2d\sigma_g\Big]\no\\
\leq& 0
\end{align}
for all time $t \geq 0$. Integrating \eqref{energy_rate} over $(0,t)$, we show
\begin{align}\label{energy_gap}
&E[u(t)]-E[u(0)]\no\\
=&\frac{n-2}{2}\int_0^t \Big[\alpha^{-1} \int_Mf(\lambda-\alpha f^{-1}R_g)^2d\mu_g+2(n-1)\beta^{-1}\int_{\partial M}h(\lambda-\beta h^{-1}h_g)^2d\sigma_g\Big]d\tau,
\end{align}
whence
$$E[u(t)]\leq E[u_0].$$

Notice that $Q(M,\pa M)$ is finite and negative implies $Y(M,\pa M)<0$. Then we claim that when $Y(M,\partial M)<0$, there exists an initial metric $g(0)=u_0^{\frac{4}{n-2}}g_0$ such that
\begin{equation}\label{initial_metric}
R_{g_0}<0 \hbox{~~in~~} M \hbox{~~and~~} h_{g_0}<0 \hbox{~~on~~} \partial M,
\end{equation}
meanwhile
\begin{equation}\label{initial_value}
E[u_0]<0.
\end{equation}
To see \eqref{initial_metric} and \eqref{initial_value}, from \cite[Lemma 1.1]{escobar3} there exists a conformal metric $g_1=u_1^{\frac{4}{n-2}}g_0$ with $R_{g_1}<0$ and $h_{g_1}=0$. Recall that there exists a constant $C_1>0$ depending only on $M,g_1,n$ such that (cf. \cite[P. 9]{escobar1} or by a contradiction argument)
$$\int_M \phi^2 d\mu_{g_1}\leq C_1\left(\int_M |\nabla \phi|_{g_1}^2d\mu_{g_1}+\int_{\pa M}\phi^2d\sigma_{g_1}\right) $$
for any $\phi \in H^1(M,g_1)$. By choosing $0<\e_0<\min\{\frac{4(n-1)}{(n-2)C_1},\min_{\bar M}(-R_{g_1})\}$,
 let $\varphi$ be the positive smooth minimizer of 
 $$\lambda_1:=\min\left\{\int_M \left(\frac{4(n-1)}{n-2}|\nabla \phi|_{g_1}^2-\e_0\phi^2 \right)d\mu_{g_1};\phi \in H^1(M,g_1) \hbox{~~and~~}\int_{\pa M}\phi^2 d\sigma_{g_1}=1\right\}.$$
Let $\bar g_0=\varphi^{\frac{4}{n-2}}g_1$, it yields $R_{\bar g_0}=(\e_0+R_{g_1})\varphi^{-\frac{4}{n-2}}<0$ and $h_{\bar g_0}=\frac{\lambda_1}{2(n-1)}\varphi^{-\frac{2}{n-2}}<0$.
Thus we obtain
\begin{align*}
E[\varphi u_1]=\int_M R_{\bar g_0} d\mu_{\bar g_0}+2(n-1)\int_{\partial M}h_{\bar g_0}d\sigma_{\bar g_0}<0.
\end{align*}
Thus we may rescale $\bar g_0$ to satisfy \eqref{conserved_quantity}
and use this rescaled metric as the initial metric.

Along the flows \eqref{mixed_cur_flows}, we obtain the uniform bounds on the volumes of $M$ and $\partial M$ with respect to the flow metric, as well as the energy $E[u(t)]$.
\begin{lemma}\label{volume_bdd}
Along the flows \eqref{mixed_cur_flows} with the normalization \eqref{conserved_quantity}, there exist two positive constants $V_0$ and $S_0$ depending on $n,f,h, Y(M,\pa M), Q(M,\partial M),E[u_0]$ such that
\begin{equation}\label{est:volumes}
V_0^{-1}\leq \im u(t)^\criti \dmz \leq V_0 \quad\hbox{~~and~~} \quad S_0^{-1}\leq  \idm u(t)^\critb \dsz \leq S_0
\end{equation}
for all $t \geq 0$. Moreover, there also exists a positive constant $C_1$ depending on $V_0$ or $S_0$, such that
$$C_1\leq E[u(t)]=\q[u(t)]  \leq E[u_0]$$
for all $t \geq 0$.
\end{lemma}
\begin{proof}
From the definitions of $Y(M,\partial M)$ and $Q(M,\partial M)$, we have
\begin{align*}
Y(M,\partial M)\Big(\int_M d\mu_g\Big)^{\frac{n-2}{n}}\leq E[u(t)]\leq E[u_0],\\
Q(M,\partial M)\Big(\int_{\partial M} d\sigma_g\Big)^{\frac{n-2}{n-1}}\leq E[u(t)]\leq E[u_0],
\end{align*}
which implies that
\begin{align*}
 \left(\int_M d\mu_g\right)^{\frac{n-2}{n}}\geq\frac{E[u_0]}{Y(M,\partial M)} \quad\mbox{ and }\quad
\Big(\int_{\partial M} d\sigma_g\Big)^{\frac{n-2}{n-1}}\geq \frac{E[u_0]}{Q(M,\partial M)},
\end{align*}
Notice that the lower bounds are positive
in view of \eqref{initial_value}. On the other hand,
\eqref{conserved_quantity} gives the required upper bounds. Thus the desired estimate \eqref{est:volumes} for volumes follows from the above estimates. Using \eqref{est:volumes} together with the above estimates, we obtain
$$Y(M,\partial M)V_0^{\frac{n-2}{n}}\leq E[u(t)]\leq E[u_0]$$
or
$$Q(M,\partial M)S_0^{\frac{n-2}{n-1}}\leq E[u(t)]\leq E[u_0],$$
which implies the second assertion.
\end{proof}

As a direct consequence of Lemma \ref{volume_bdd},
we can deduce from \eqref{eqn:constrant_time_funcs}, \eqref{conserved_quantity} that there exist positive uniform constants $\alpha_i$, $\beta_i$ and $\lambda_i, i=1,2$ such that
\begin{equation}\label{alpha_beta_lambda_bdd}
\alpha_1\leq \alpha(t)\leq\alpha_2\hspace{2mm}\mbox{ and }\hspace{2mm}\beta_1\leq \beta(t)\leq\beta_2\hspace{2mm}\mbox{ and }\hspace{2mm}\lambda_1=\lambda(0)\leq \lambda(t)\leq\lambda_2
\end{equation}
for all $t\geq 0$.

\section{Short time existence}\label{Sec3}\label{Sect4}

The proof of the short time existence mainly follows the strategy of Brendle \cite{Brendle2}. We first have the following estimates for linear parabolic equations with constant coefficients.
\begin{lemma}\label{lem2.1}
The linear partial differential equation
$$\displaystyle\frac{4}{n-2}\partial_tu(x',x_n,t)-\frac{4(n-1)}{n-2}\sum_{i=1}^n\frac{\partial^2}{\partial x_i^2}u(x',x_n,t)=F(x',x_n,t)$$
for $x_n\geq 0$ with the boundary condition
$$\displaystyle\frac{4}{n-2}\partial_tu(x',0,t)-\frac{2}{n-2}\frac{\partial}{\partial x_n}u(x',0,t)=G(x',t)$$
for $t\geq 0$ and the initial condition
$$u(x',x_n,0)=0$$
for $x_n\geq 0$ has a unique solution. The solution satisfies the estimates
\begin{equation*}
\iint\left|\frac{\partial}{\partial x_i}u(x',0,t)\right|^pdx'dt\leq C \iiint|F(x',x_n,t)|^pdx'dx_ndt
+C\iint|G(x',t)|^pdx'dt
\end{equation*}
and
\begin{equation*}
\iint\left|\partial_tu(x',0,t)\right|^pdx'dt\leq C \iiint|F(x',x_n,t)|^pdx'dx_ndt
+C\iint|G(x',t)|^pdx'dt.
\end{equation*}
\end{lemma}

\begin{lemma}\label{lem2.2}
Let $u$ be the solution of the linear partial differential equation
$$\displaystyle\frac{4}{n-2}\partial_tu(x',x_n,t)-\frac{4(n-1)}{n-2}\sum_{i=1}^n\frac{\partial^2}{\partial x_i^2}u(x',x_n,t)=F(x',x_n,t)$$
for $x_n\geq 0$ with the boundary condition
$$\displaystyle\frac{4}{n-2}\partial_tu(x',0,t)-\frac{2}{n-2}\frac{\partial}{\partial x_n}u(x',0,t)=G(x',t)$$
for $t\geq 0$ and the initial condition
$$u(x',x_n,0)=0$$
for $x_n\geq 0$. Then we have
\begin{equation*}
\begin{split}
\int_0^T\iint\left|\frac{\partial}{\partial x_i}u(x',x_n,t)\right|^pdx'dx_ndt&\leq C T^{\frac{1}{2}}\iint|F(x',x_n,t)|^pdx'dx_ndt\\
&\hspace{4mm}
+CT^{\frac{1}{2}}\int_0^T\int|G(x',t)|^pdx'dt
\end{split}
\end{equation*}
and
\begin{equation*}
\begin{split}
\int_0^T\iint\left|\partial_tu(x',x_n,t)\right|^pdx'dx_ndt&\leq C \int_0^T\iint|F(x',x_n,t)|^pdx'dx_ndt\\
&\hspace{4mm}
+C\int_0^T\int|G(x',t)|^pdx'dt.
\end{split}
\end{equation*}
\end{lemma}

For the proof, see \cite[Lemmas 2.1 and 2.2]{Brendle2}.
Using Lemmas \ref{lem2.1} and \ref{lem2.2}, we can prove the following lemma through the freezing coefficients method and a covering argument,
which reduce the case of variable coefficients to the case of
constant coefficients.

\begin{lemma}\label{lem2.3}
Let $v$ be a H\"older continuous function in $\overline M \times [0,T]$, which is bounded below by a positive constant,
$f$ and $h$ be smooth negative functions defined in $M$ and on $\partial M$ respectively.
If $\alpha(v)$ and $\beta(v)$ are constants depending on $v$ which are defined as
$$\alpha(v)=\frac{1}{a}\Big(\int_M -fv^{\frac{2n}{n-2}}d\mu_{g_0}\Big)^{\frac{2}{n}}, \quad \beta(v)=\frac{1}{b}\Big(\int_{\partial M}-hv^{\frac{2(n-1)}{n-2}}d\sigma_{g_0}\Big)^{\frac{1}{n-1}},$$
then the partial differential equation
\begin{equation*}
\frac{4}{n-2}\alpha(v)^{-1}fv^{\frac{4}{n-2}}\partial_t u+\frac{4(n-1)}{n-2}\Delta_{g_0}u=\alpha(v)^{-1}f F
\end{equation*}
in $M$ with boundary condition
\begin{equation*}
\frac{4}{n-2}\beta(v)^{-1}hv^{\frac{2}{n-2}}\partial_t u-\frac{2}{n-2}\frac{\partial u}{\partial \nu_0}=\beta(v)^{-1}h G
\end{equation*}
 on $\partial M$ and the initial condition
$$u(0)=u_0$$
has a unique solution on some small interval $0\leq t\leq T$.
The solution satisfies the estimates
\begin{equation*}
\begin{split}
\int_0^T\int_M\left|\partial_tu\right|^pd\mu_{g_0}dt
\leq C\int_0^T\int_M|F|^pd\mu_{g_0}dt+C\int_0^T\int_{\partial M}|G|^pd\sigma_{g_0}dt
\end{split}
\end{equation*}
and
\begin{equation*}
\begin{split}
\int_0^T\int_M\left|\frac{\partial}{\partial x_i}u\right|^pd\mu_{g_0}dt
\leq CT^{\frac{1}{2}}\int_0^T\int_M|F|^pd\mu_{g_0}dt+CT^{\frac{1}{2}}\int_0^T\int_{\partial M}|G|^pd\sigma_{g_0}dt.
\end{split}
\end{equation*}
Furthermore, we have
\begin{equation*}
\begin{split}
\int_0^T\int_{\partial M}\left|\partial_tu\right|^pd\sigma_{g_0}dt
\leq C\int_0^T\int_M|F|^pd\mu_{g_0}dt+C\int_0^T\int_{\partial M}|G|^pd\sigma_{g_0}dt
\end{split}
\end{equation*}
and
\begin{equation*}
\begin{split}
\int_0^T\int_{\partial M}\left|\frac{\partial}{\partial x_i}u\right|^pd\sigma_{g_0}dt
\leq CT^{\frac{1}{2}}\int_0^T\int_M|F|^pd\mu_{g_0}dt+CT^{\frac{1}{2}}\int_0^T\int_{\partial M}|G|^pd\sigma_{g_0}dt.
\end{split}
\end{equation*}
\end{lemma}

\begin{lemma}\label{lem2.4}
Let $u$ be the solution of the partial differential equation
$$\frac{4}{n-2}v^{\frac{4}{n-2}}\partial_t u+\frac{4(n-1)}{n-2}\alpha(v) f^{-1}\Delta_{g_0}u=F$$
in $M$ with boundary condition
$$\frac{4}{n-2}v^{\frac{2}{n-2}}\partial_t u-\frac{2}{n-2}\beta(v)h^{-1}\frac{\partial u}{\partial \nu_0}=G$$
 on $\partial M$. If
\begin{equation}\label{L_p_assumption}
\begin{split}
&D^mF\in L^p(M\times [\delta,T]),\\
&D^mG\in L^p(\partial M\times [\delta,T]),\\
&D^mu\in L^p(M\times [\delta,T])\cap L^p(\partial M\times [\delta,T]),\\
&D^m v\in L^p(M\times [\delta,T])\cap L^p(\partial M\times [\delta,T])
\end{split}
\end{equation}
for all $p$, then we have
$$D^{m+1}u\in L^p(M\times [2\delta,T])\cap L^p(\partial M\times [2\delta,T])$$
for all $p$.
\end{lemma}
\begin{proof}
To prove the assertion for $m=0$, we choose a cut-off function $\eta$ such that
$\eta(t)=0$ for $0\leq t\leq\delta$ and $\eta(t)=1$ for $t\geq 2\delta$. Then the function $\eta u$ satisfies
\begin{equation*}
\frac{4}{n-2}v^{\frac{4}{n-2}}\partial_t (\eta u)+\frac{4(n-1)}{n-2}\alpha(v) f^{-1}\Delta_{g_0}(\eta u)\\
=\eta F+\frac{4}{n-2}v^{\frac{4}{n-2}}u\partial_t\eta
\end{equation*}
in $M$ with boundary condition
\begin{equation*}
\frac{4}{n-2}v^{\frac{2}{n-2}}\partial_t(\eta u)-\frac{2}{n-2}\beta(v)h^{-1}\frac{\partial}{\partial \nu_0}(\eta u)=\eta G+\frac{4}{n-2}v^{\frac{2}{n-2}}u\partial_t\eta
\end{equation*}
on $\partial M$.
Since
$$\eta F+\frac{4}{n-2}v^{\frac{4}{n-2}}u\partial_t\eta\in L^p(\partial M\times [\delta,T])$$
and
$$\eta G+\frac{4}{n-2}v^{\frac{2}{n-2}}u\partial_t\eta\in L^p(\partial M\times [\delta,T])$$
by \eqref{L_p_assumption}, it follows from
Lemma \ref{lem2.3} that
$$D(\eta u)\in L^p(M\times [\delta,T])\cap L^p(\partial M\times [\delta,T]),$$
which implies that
$$D u\in L^p(M\times [2\delta,T])\cap L^p(\partial M\times [2\delta,T])$$
since $\eta(t)=1$ for $t\geq 2\delta$.

Now we assume that the assertion holds for $m$. The function $\partial_t u$ satisfies
\begin{equation*}
\begin{split}
&\frac{4}{n-2}v^{\frac{4}{n-2}}\partial_t (\partial_t u)+\frac{4(n-1)}{n-2}\alpha(v) f^{-1}\Delta_{g_0}(\partial_t u)\\
&=\partial_t F-\frac{16}{(n-2)^2}v^{\frac{6-n}{n-2}}\partial_t v\partial_t u
-\frac{4(n-1)}{n-2}\partial_t\alpha(v) f^{-1}\Delta_{g_0}u\\
&=\partial_t F-\frac{16}{(n-2)^2}v^{\frac{6-n}{n-2}}\partial_t v\partial_t u
+\frac{\partial_t\alpha(v)}{\alpha(v)}\Big(\frac{4}{n-2}v^{\frac{4}{n-2}}\partial_t u-F\Big)
\end{split}
\end{equation*}
in $M$
with boundary condition
\begin{equation*}
\begin{split}
&\frac{4}{n-2}v^{\frac{2}{n-2}}\partial_t(\partial_t u)-\frac{2}{n-2}\beta(v)h^{-1}\frac{\partial}{\partial \nu_0}(\partial_t u)\\
&=\partial_t G+\frac{8}{(n-2)^2}v^{\frac{4-n}{n-2}}\partial_t v\partial_t u
+\frac{2}{n-2}\partial_t\beta(v)h^{-1}\frac{\partial u}{\partial\nu_0}\\
&=\partial_t G+\frac{8}{(n-2)^2}v^{\frac{4-n}{n-2}}\partial_t v\partial_t u
+\frac{\partial_t\beta(v)}{\beta(v)}\Big(\frac{4}{n-2}v^{\frac{2}{n-2}}\partial_t u-G\Big)
\end{split}
\end{equation*}
on $\partial M$.
Note that
\begin{equation*}
\frac{\partial_t\alpha(v)}{\alpha(v)}=\frac{2}{n}\left(\int_M-fv^{\frac{2n}{n-2}}d\mu_{g_0}\right)^{-1}\int_M-fv^{\frac{n+2}{n-2}}\partial_t vd\mu_{g_0}
\end{equation*}
and
\begin{equation*}
\frac{\partial_t\beta(v)}{\beta(v)}=\frac{2}{n-1}\left(\int_{\partial M}-hv^{\frac{2(n-1)}{n-2}}d\sigma_{g_0}\right)^{-1}\int_{\partial M}-hv^{\frac{n}{n-2}}\partial_t vd\sigma_{g_0},
\end{equation*}
which implies that
$$\left|\partial_t^{m-1}\Big(\frac{\partial_t\alpha(v)}{\alpha(v)}\Big)\right|\leq C
\quad
\mbox{ and }\quad \left|\partial_t^{m-1}\Big(\frac{\partial_t\beta(v)}{\beta(v)}\Big)\right|\leq C$$
by \eqref{L_p_assumption} and the fact that $v$ is bounded below by a positive constant.
This together with \eqref{L_p_assumption} imply that
\begin{equation*}
\begin{split}
&D^{m-1}\left(\partial_t F-\frac{16}{(n-2)^2}v^{\frac{6-n}{n-2}}\partial_t v\partial_t u
+\frac{\partial_t\alpha(v)}{\alpha(v)}\Big(\frac{4}{n-2}v^{\frac{4}{n-2}}\partial_t u-F\Big)\right)\in L^p(M\times [\delta,T]),\\
&D^{m-1}\left(\partial_t G+\frac{8}{(n-2)^2}v^{\frac{4-n}{n-2}}\partial_t v\partial_t u
+\frac{\partial_t\beta(v)}{\beta(v)}\Big(\frac{4}{n-2}v^{\frac{2}{n-2}}\partial_t u-G\Big)\right)\in L^p(\partial M\times [\delta,T]),\\
&D^{m-1}\partial_t u\in L^p(M\times [\delta,T])\cap L^p(\partial M\times [\delta,T]),\\
&D^{m-1} v\in L^p(M\times [\delta,T])\cap L^p(\partial M\times [\delta,T]),
\end{split}
\end{equation*}
for all $p$. From this, it follows from Lemma \ref{lem2.3} that
$$D^m\partial_t u\in  L^p(M\times [2\delta,T])\cap L^p(\partial M\times [2\delta,T])$$
for all $p$. Similar estimates hold for the spatial derivatives. Thus we conclude that
$$D^{m+1} u\in  L^p(M\times [2\delta,T])\cap L^p(\partial M\times [2\delta,T])$$
as required.
\end{proof}

Now we use the contraction mapping principle to prove the short time existence of
the flows (\ref{mixed_cur_flows}).

\begin{theorem}\label{thm:short time}
The partial differential equation
$$\frac{4}{n-2}u^{\frac{4}{n-2}}\partial_t u+\frac{4(n-1)}{n-2}\alpha(u) f^{-1}\Delta_{g_0}u=
\alpha(u) f^{-1}R_{g_0}u
-\lambda(u)u^{\frac{n+2}{n-2}}$$
in $M$ with boundary condition
$$\frac{4}{n-2}u^{\frac{2}{n-2}}\partial_t u-\frac{2}{n-2}\beta(u)h^{-1}\frac{\partial u}{\partial \nu_0}=
\beta(u)h^{-1}h_{g_0}u-\lambda(u)u^{\frac{n}{n-2}}$$
 on $\partial M$ and the initial condition
$$u=u_0$$
for $t=0$ has a unique solution on some small time interval $0\leq t\leq T$,
where $\lambda(u)$ is defined as
\begin{equation*}
\lambda(u)=
-\frac{\int_M (\tfrac{4(n-1)}{n-2}|\nabla u|_{g_0}^2 +R_{g_0}u^2) d\mu_{g_0}+2(n-1)\int_{\partial M}h_{g_0}u^2 d\sigma_{g_0}
}{a\Big(\int_M -f u^{\frac{2n}{n-2}}d\mu_{g_0}\Big)^{\frac{n-2}{n}}+2(n-1)b\Big(\int_{\partial M}-h u^{\frac{2(n-1)}{n-2}}d\sigma_{g_0}\Big)^{\frac{n-2}{n-1}}}.
\end{equation*}
Moreover, the solution is of class
$C^{0,\gamma}$ for $0\leq t\leq T$ and smooth for $0<t\leq T$.
\end{theorem}
\begin{proof}
For  any $p>n$, we define the set $E$ by
$$E=\{v\in C^0([0,T];H^1(M)); v(0,x)=u_0, \|v\|_{W^{1,p}(M\times[0,T])}+\|v\|_{W^{1,p}(\partial M\times[0,T])}\leq 1\}$$
and put the distance by
$$\|v-\tilde{v}\|:=\|v-\tilde{v}\|_{W^{1,p}(M\times[0,T])}+\|v-\tilde{v}\|_{W^{1,p}(\partial M\times[0,T])}, \hbox{~~for~~} v,\tilde{v}\in E.$$
It is routine to check that $E$ is a complete metric space.
For $v\in E$, we consider
$$\frac{4}{n-2}v^{\frac{4}{n-2}}\partial_t u+\frac{4(n-1)}{n-2}\alpha(v) f^{-1}\Delta_{g_0}u=\alpha(v) f^{-1}R_{g_0}v-\lambda(v)v^{\frac{n+2}{n-2}}
$$
in $M$ with boundary condition
$$\frac{4}{n-2}v^{\frac{2}{n-2}}\partial_t u-\frac{2}{n-2}\beta(v)h^{-1}\frac{\partial u}{\partial \nu_0}=
\beta(v)h^{-1}h_{g_0}v-\lambda(v)v^{\frac{n}{n-2}}$$
 on $\partial M$ and the initial condition
$$u=u_0$$
for $t=0$. Since
$$\|v\|_{W^{1,p}(M\times[0,T])}+\|v\|_{W^{1,p}(\partial M\times[0,T])}\leq 1,$$
the Sobolev inequality implies that
$$|v|\leq C$$
for some constant $C$ independent of $v$.
Thus, applying Lemma \ref{lem2.3} with
$$F=\alpha(v) f^{-1}R_{g_0}v
-\lambda(v)v^{\frac{n+2}{n-2}}\mbox{
and }G=\beta(v)h^{-1}h_{g_0}v-\lambda(v)v^{\frac{n}{n-2}},$$
we obtain
$$\|u\|_{W^{1,2p}(M\times[0,T])}+\|u\|_{W^{1,2p}(\partial M\times[0,T])}\leq C.$$
This together with the H\"{o}lder's inequality implies that
\begin{equation*}
\begin{split}
\|u\|_{W^{1,p}(M\times[0,T])}+\|u\|_{W^{1,p}(\partial M\times[0,T])}
&\leq C\|u\|_{W^{1,2p}(M\times[0,T])}T^{\frac{1}{2p}}+C\|u\|_{W^{1,2p}(M\times[0,T])}T^{\frac{1}{2p}}\\
&\leq CT^{\frac{1}{2p}}.
\end{split}
\end{equation*}
Thus, we can conclude that
$$\|u\|_{W^{1,p}(M\times[0,T])}+\|u\|_{W^{1,p}(\partial M\times[0,T])}\leq 1$$
if $T$ is sufficiently small.
Hence, we can define a map
$$\mathcal{F}:E\to E$$
by $\mathcal{F}(v)=u.$

Put $\mathcal{F}(v)=u$ and $\mathcal{F}(\tilde{v})=\tilde{u}$. Then the function $u-\tilde{u}$ satisfies
\begin{equation}\label{2.5}
\begin{split}
&\frac{4}{n-2}v^{\frac{4}{n-2}}\partial_t (u-\tilde{u})+\frac{4(n-1)}{n-2}\alpha(v) f^{-1}\Delta_{g_0}(u-\tilde{u})\\
=&\alpha(v) f^{-1}R_{g_0}v-\alpha(\tilde{v}) f^{-1}R_{g_0}\tilde{v}
-\lambda(v)v^{\frac{n+2}{n-2}}+\lambda(\tilde{v})\tilde{v}^{\frac{n+2}{n-2}}+\frac{4}{n-2}(v^{\frac{4}{n-2}}-\tilde{v}^{\frac{4}{n-2}})\partial_t \tilde{u}\\
&+\frac{\alpha(v)-\alpha(\tilde{v}))}{\alpha(\tilde u)}\left[\alpha(\tilde u) f^{-1}R_{g_0}\tilde u
-\lambda(\tilde u)\tilde u^{\frac{n+2}{n-2}}-\frac{4}{n-2}\tilde u^{\frac{4}{n-2}}\partial_t \tilde u\right]:=F
\end{split}
\end{equation}
in $M$,
\begin{equation}\label{2.6}
\begin{split}
&\frac{4}{n-2}v^{\frac{2}{n-2}}\partial_t (u-\tilde{u})-\frac{2}{n-2}\beta(v)h^{-1}\frac{\partial}{\partial \nu_0}(u-\tilde{u})\\
&=\beta(v)h^{-1}h_{g_0}v-\beta(\tilde{v})h^{-1}h_{g_0}\tilde{v}-\lambda(v)v^{\frac{n}{n-2}}+\lambda(\tilde{v})\tilde{v}^{\frac{n}{n-2}}\\
&\hspace{4mm}+\frac{4}{n-2}(v^{\frac{2}{n-2}}-\tilde{v}^{\frac{2}{n-2}})\partial_t\tilde{u}-\frac{2}{n-2}(\beta(v)-\beta(\tilde{v}))h^{-1}\frac{\partial}{\partial \nu_0}\tilde{u}:=G
\end{split}
\end{equation}
on $\partial M$ and
$$u-\tilde{u}=0$$
for $t=0$.
Then the functions $F$ and $G$ satisfy
$$\|F\|_{L^{2p}(M\times[0,T])}\leq C(\|v-\tilde{v}\|_{W^{1,p}(M\times [0,T])}+\|v-\tilde{v}\|_{W^{1,p}(\partial M\times [0,T])})$$
and
$$\|G\|_{L^{2p}(\partial M\times[0,T])}\leq C(\|v-\tilde{v}\|_{W^{1,p}(M\times [0,T])}+\|v-\tilde{v}\|_{W^{1,p}(\partial M\times [0,T])}).$$
Hence, we can apply Lemma \ref{lem2.3} to \eqref{2.5} and \eqref{2.6}
to get
\begin{equation*}
\|u-\tilde{u}\|_{W^{1,2p}(M\times [0,T])}+\|u-\tilde{u}\|_{W^{1,2p}(\partial M\times [0,T])}
\leq C(\|v-\tilde{v}\|_{W^{1,p}(M\times [0,T])}+\|v-\tilde{v}\|_{W^{1,p}(\partial M\times [0,T])}).
\end{equation*}
This together with the H\"{o}lder's inequality imply that
\begin{align*}
&\|u-\tilde{u}\|_{W^{1,p}(M\times [0,T])}+\|u-\tilde{u}\|_{W^{1,p}(\partial M\times [0,T])}\\
\leq& CT^{\frac{1}{2p}}(\|v-\tilde{v}\|_{W^{1,p}(M\times [0,T])}+\|v-\tilde{v}\|_{W^{1,p}(\partial M\times [0,T])}).
\end{align*}
Thus, we can conclude that
\begin{equation*}
\|u-\tilde{u}\|_{W^{1,p}(M\times [0,T])}+\|u-\tilde{u}\|_{W^{1,p}(\partial M\times [0,T])}
\leq \frac{1}{2}(\|v-\tilde{v}\|_{W^{1,p}(M\times [0,T])}+\|v-\tilde{v}\|_{W^{1,p}(\partial M\times [0,T])})
\end{equation*}
if $T$ is sufficiently small. The above means that $\mathcal{F}$ is a contraction mapping. Therefore, the contraction mapping principle implies that
$\mathcal{F}$ has a unique fixed point $u\in E$. It follows from Lemma \ref{lem2.4}
that $u$ is smooth for $0<t\leq T$. Since $\mathcal{F}(u)=u$, the function $u$ is the
desired solution of the flows \eqref{mixed_cur_flows_u} by \eqref{evolution_u}.
\end{proof}

\section{Uniform bounds of conformal factor}\label{Sect5}

This section is devoted to the proof of uniform bounds of the conformal factor along the flows.

\begin{lemma}\label{lower_bdd_u}
Along the flows \eqref{mixed_cur_flows}, the conformal factor $u$ satisfies
\begin{equation}\label{1.22}
\min_{\overline{M}} u(t)^{\frac{n+2}{n-2}}\geq\min\left\{\min_{\overline{M}}u(0)^{\frac{n+2}{n-2}},
\left(\frac{\alpha_1}{\lambda_2}\min_M (f^{-1}R_{g_0})\right)^{\frac{n+2}{4}},
\left(\frac{\beta_1}{\lambda_2}\min_{\partial M} (h^{-1}h_{g_0})\right)^{\frac{n+2}{2}}\right\}.
\end{equation}
Here, $\alpha_i$, $\beta_i$ and $\lambda_i, i=1,2$ are the positive constants given in \eqref{alpha_beta_lambda_bdd}.
\end{lemma}

Before proving Lemma \ref{lower_bdd_u}, we would like to point out that
the constant on the right hand side of \eqref{1.22} is positive,
thanks to \eqref{initial_metric} and the assumption that $f<0$ and $h<0$.

\begin{proof}
If the flows \eqref{mixed_cur_flows} exist on $[0,T)$, i.e.
a positive solution $u$ exists for \eqref{evolution_u} on $[0,T)$, then for $\epsilon> 0$, we define $F(x,t)=u(x,t)^{\frac{n+2}{n-2}}+\epsilon(1 + t)$ on $\overline{M}\times[0,T)$. For abbreviation, we
denote by $C$ the constant on the right-hand side of \eqref{1.22}.
By contradiction, we suppose that
\begin{equation}\label{1.23}
F(x,t)\leq C
\end{equation}
for some $x\in {\overline{M}}$ and $t=t_0$. Then $t_0>0$ because $F=u(0)^{\frac{n+2}{n-2}}+\epsilon>\displaystyle\min_{\overline{M}}u(0)^{\frac{n+2}{n-2}}$
at $t=0$. We may assume that $t_0$ is the smallest $t$ which satisfies (\ref{1.23}). Then we
have
\begin{equation}\label{1.24}
F(x, t)>C\mbox{ for }t \in [0, t_0), F(x, t_0) \geq C\mbox{~~and~~}F(x_0, t_0) = C\mbox{~~ for some }x_0 \in \overline{M}.
\end{equation}
This implies that
\begin{equation}\label{1.25}
u(x_0,t_0)=\min_{x\in\overline{M}}u(x,t_0).
\end{equation}
We have the following two cases according to $x_0\in M$ and $x_0\in\partial M$.\\
Case (i). If $x_0\in M$, then
by the fact that $F(x_0, t_0) = C$
we have
\begin{equation}\label{1.26}
u(x_0,t_0)^{\frac{4}{n-2}}\leq\frac{\alpha_1}{\lambda_2}\min_M (f^{-1}R_{g_0}).
\end{equation}
Therefore, at $(x_0,t_0)$, we have
\begin{align*}
0\geq\frac{\partial F}{\partial t}&=\epsilon+\frac{\partial}{\partial t}\Big(u^{\frac{n+2}{n-2}}\Big)\\
&=\epsilon+\frac{n+2}{4}\Big(
-\frac{4(n-1)}{n-2}\alpha(t) f^{-1}\Delta_{g_0}u+\alpha(t) f^{-1}R_{g_0}u
-\lambda(t)u^{\frac{n+2}{n-2}}\Big)\\
&\geq\epsilon+\frac{n+2}{4}\Big(\alpha(t) f^{-1}R_{g_0}u
-\lambda(t)u^{\frac{n+2}{n-2}}\Big)\geq \epsilon,
\end{align*}
where the first inequality follows from \eqref{1.24},
the second equality follows from \eqref{evolution_u},
the second inequality follows from \eqref{alpha_beta_lambda_bdd}, \eqref{1.25} and the fact that $f<0$,
and the last inequality follows from \eqref{1.26} and
the fact that $u>0$.
This contradicts $\epsilon> 0$.\\
Case (ii). If $x_0\in \partial M$, then
by the fact that $F(x_0, t_0) = C$
we have
\begin{equation}\label{1.27}
u(x_0,t_0)^{\frac{2}{n-2}}\leq\frac{\beta_1}{\lambda_2}\min_{\partial M} (h^{-1}h_{g_0}).
\end{equation}
Therefore, at $(x_0,t_0)$, we have
\begin{align*}
0\geq\frac{\partial F}{\partial t}&=\epsilon+\frac{\partial}{\partial t}\Big(u^{\frac{n+2}{n-2}}\Big)\\
&=\epsilon+\frac{n+2}{4}u^{\frac{2}{n-2}}\Big(\frac{2}{n-2}\beta(t)h^{-1}\frac{\partial}{\partial \nu_0}u+\beta(t)h^{-1}h_{g_0}u-\lambda(t)u^{\frac{n}{n-2}}\Big)\\
&\geq \epsilon+\frac{n+2}{4}u^{\frac{2}{n-2}}\Big(\beta(t)h^{-1}h_{g_0}u-\lambda(t)u^{\frac{n}{n-2}}\Big)\geq \epsilon,
\end{align*}
where the first inequality follows from \eqref{1.24},
the second equality follows from \eqref{evolution_u},
the second inequality follows from
\eqref{alpha_beta_lambda_bdd}, \eqref{1.25} and the fact that $h<0$,
and the last equality follows from \eqref{1.27} and
the fact that $u>0$.
This again contradicts $\epsilon> 0$.

Therefore, we have $F(x, t)>C$ and  Lemma \ref{lower_bdd_u} follows by letting $\epsilon \to 0$.
\end{proof}

Similarly, we can prove the following:

\begin{lemma}\label{upper_bdd_u}
Along the flows \eqref{mixed_cur_flows}, the conformal factor $u$ satisfies
\begin{equation*}
\max_{\overline{M}} u(t)^{\frac{n+2}{n-2}}\leq\max\left\{\max_{\overline{M}}u(0)^{\frac{n+2}{n-2}},
\left(\frac{\alpha_2}{\lambda_1}\max_M (f^{-1}R_{g_0})\right)^{\frac{n+2}{4}},
\left(\frac{\beta_2}{\lambda_1}\max_{\partial M} (h^{-1}h_{g_0})\right)^{\frac{n+2}{2}}\right\}.
\end{equation*}
\end{lemma}

We skip the proof of Lemma \ref{upper_bdd_u}, since it is almost identical to the one of Lemma \ref{lower_bdd_u},
except we choose the auxiliary function $F(x,t)=u(x,t)^{\frac{n+2}{n-2}}-\epsilon(1 + t)$ in this case.

\section{Long time existence and asymptotic convergence}\label{Sect6}

From \eqref{mixed_cur_flows_u}, \eqref{eq:sc} and \eqref{eq:mc}, we obtain
\begin{align}\label{R_t}
\partial_t R_g=\partial_t\Big[u^{-\frac{n+2}{n-2}}L_{g_0}(u)\Big]
=&-\frac{n+2}{n-2}u^{-\frac{2n}{n-2}}u_t L_{g_0}(u)+u^{-\frac{n+2}{n-2}}L_{g_0}(u_t)\no\\
=&-\frac{n+2}{4}R_g(\alpha f^{-1}R_g-\lambda)+\frac{n-2}{4}u^{-\frac{n+2}{n-2}}L_{g_0}((\alpha f^{-1}R_g-\lambda)u)\no\\
=&-\frac{n+2}{4}R_g(\alpha f^{-1}R_g-\lambda)+\frac{n-2}{4}L_g(\alpha f^{-1}R_g-\lambda)\no\\
=&-(n-1)\Delta_g(\alpha f^{-1}R_g-\lambda)-R_g(\alpha f^{-1}R_g-\lambda).
\end{align}
and
\begin{align}\label{h_t}
\partial_t h_g=\frac{2}{n-2}\partial_t\Big[u^{-\frac{n}{n-2}}B_{g_0}(u)\Big]
=&\frac{2}{n-2}\Big[-\frac{n}{n-2}u^{-\frac{2(n-1)}{n-2}}u_tB_{g_0}(u)+u^{-\frac{n}{n-2}}B_{g_0}(u_t)\Big]\no\\
=&-\frac{n}{4}(\beta h^{-1}h_g-\lambda)h_g+\frac{1}{2}u^{-\frac{n}{n-2}}B_{g_0}((\beta h^{-1}h_g-\lambda)u)\no\\
=&-\frac{n}{4}(\beta h^{-1}h_g-\lambda)h_g+\frac{1}{2}B_g(\beta h^{-1}h_g-\lambda)\no\\
=&\frac{1}{2}\frac{\partial}{\partial \nu_g}(\beta h^{-1}h_g-\lambda)-\frac{1}{2}(\beta h^{-1}h_g-\lambda)h_g,
\end{align}
where we have used conformal covariance of both operators $L_{g_0}$ and $B_{g_0}$ that
$$L_{g_0}(u\varphi)=u^{\frac{n+2}{n-2}}L_g(\varphi)\quad \hbox{~and~} \quad B_{g_0}(u\varphi)=u^{\frac{n }{n-2}}B_g(\varphi) \hbox{~~for~~} \varphi \in C^\infty(M).$$
It follows from \eqref{eqn:constrant_time_funcs} and \eqref{mixed_cur_flows_u} that
\begin{equation}\label{alpha_t}
\partial_t\alpha
=-\alpha\Big(\int_M -fd\mu_g\Big)^{-1}\int_M f(\alpha f^{-1} R_g-\lambda)d\mu_g
\end{equation}
and
\begin{equation}\label{beta_t}
\partial_t\beta=-\frac{\beta}{2}\Big(\int_{\partial M}-hd\sigma_g\Big)^{-1}\int_{\partial M}h(\beta h^{-1}h_g-\lambda)d\sigma_g.
\end{equation}
On the other hand, it follows from
\eqref{eqn:constrant_time_funcs}, \eqref{conserved_quantity}
and \eqref{energy_rate} that
\begin{align}\label{lambda_t}
\partial_t\lambda=&-\frac{d}{dt}E[u(t)]\no\\
=&\frac{n-2}{2}\left[\alpha^{-1} \int_M-f(\alpha f^{-1}R_g-\lambda)^2d\mu_g+2(n-1)\beta^{-1}\int_{\partial M}-h(\beta h^{-1}h_g-\lambda)^2d\sigma_g\right]\no\\
\geq& 0.
\end{align}
Therefore,  we have in $M$
\begin{align}\label{v_t_1}
&\partial_t (\alpha f^{-1}R_g-\lambda)\no\\
=&\alpha f^{-1}\partial_t R_g+[(\alpha f^{-1}R_g-\lambda)+\lambda]\alpha^{-1}\partial_t\alpha -\partial_t \lambda\no\\
=&-(n-1)\alpha f^{-1}\Delta_g(\alpha f^{-1}R_g-\lambda)-(\alpha f^{-1}R_g-\lambda)^2
-\lambda(\alpha f^{-1}R_g-\lambda)\no\\
&+[(\alpha f^{-1}R_g-\lambda)+\lambda]\alpha^{-1}\partial_t\alpha -\partial_t \lambda\no\\
\leq&-(n-1)\alpha f^{-1}\Delta_g(\alpha f^{-1}R_g-\lambda)-(\alpha f^{-1}R_g-\lambda)^2
-\lambda(\alpha f^{-1}R_g-\lambda)\no\\
&+[(\alpha f^{-1}R_g-\lambda)+\lambda]\Big(\int_M -fd\mu_g\Big)^{-1}\Big(\int_M -f(\alpha f^{-1} R_g-\lambda)d\mu_g\Big),
\end{align}
where the second equality follows from \eqref{R_t}, and the last inequality follows from \eqref{alpha_t} and \eqref{lambda_t}.
On the other hand,  we have  on $\partial M$
\begin{align}\label{v_t_2}
&\partial_t (\beta h^{-1}h_g-\lambda)\no\\
=&\beta h^{-1}\partial_t h_g+[(\beta h^{-1}h_g-\lambda)+\lambda]\beta^{-1}\partial_t\beta -\partial_t \lambda\no\\
=&\frac{1}{2}\beta h^{-1}\frac{\partial}{\partial \nu_g}(\beta h^{-1}h_g-\lambda)-\frac{1}{2}(\beta h^{-1}h_g-\lambda)^2-\frac{1}{2}\lambda(\beta h^{-1}h_g-\lambda)\no\\
&+[(\beta h^{-1}h_g-\lambda)+\lambda]\beta^{-1}\partial_t\beta -\partial_t \lambda\no\\
\leq& \frac{1}{2}\beta h^{-1}\frac{\partial}{\partial \nu_g}(\beta h^{-1}h_g-\lambda)-\frac{1}{2}(\beta h^{-1}h_g-\lambda)^2-\frac{1}{2}\lambda(\beta h^{-1}h_g-\lambda)\no\\
&+\frac{1}{2}[(\beta h^{-1}h_g-\lambda)+\lambda]\Big(\int_{\partial M}-hd\sigma_g\Big)^{-1}\Big(\int_{\partial M}-h(\beta h^{-1}h_g-\lambda)d\sigma_g\Big),
\end{align}
where the second equality follows from \eqref{h_t}, and the last inequality follows from \eqref{alpha_t} and \eqref{lambda_t}.

Let
$v=\displaystyle\frac{4}{n-2}\partial_t\log u.$ Then it follows from \eqref{mixed_cur_flows_u} that
\begin{equation}\label{v}
\begin{split}
\left\{\begin{array}{ll}
v=\alpha f^{-1}R_g-\lambda &\hspace{2mm}\mbox{ in }M,\\
v=\beta h^{-1}h_g-\lambda&\hspace{2mm}\mbox{ on }\partial M.
\end{array}
\right.
\end{split}
\end{equation}
As in \cite{ChenHo}, we can show that scalar curvature and mean curvature are uniformly bounded below.

\begin{lemma}\label{lem:v_bdd}
Along the flows \eqref{mixed_cur_flows}, there holds\footnote{The proof here requires $C^0$-regularity of $R_{g(t)}$ and $h_{g(t)}$ near $t=0$, which is lack due to Theorem \ref{thm:short time}. However, the same argument works well starting from half of the maximal existence time of the flows instead of $t=0$.}
\begin{equation*}
v(x,t)\leq \max\left\{\sup_M\big(\alpha(0) f^{-1}R_{g_0}-\lambda(0)\big),\sup_{\partial M}\big(\beta(0) h^{-1}h_{g_0}-\lambda(0)\big),0\right\}:=C_0
\end{equation*}
for all $x \in \overline M$ and $t \geq 0$.
\end{lemma}
\begin{proof}
For $\epsilon>0$, let $F=v-\epsilon(t+1)$ on $\overline M\times [0,\infty)$. We claim that
$F<C_0$ on $\overline M\times [0,\infty)$.
By contradiction, we suppose that
\begin{equation}\label{4.10}
F(x_0,t_0)\geq C_0\mbox{ for some }(x_0,t_0)\in \overline{M}\times [0,\infty).
\end{equation}
Note that $t_0>0$ since
$$F(x,0)=v(x,0)-\epsilon=\alpha(0) f^{-1}R_{g_0}-\lambda(0)-\epsilon<C_0\mbox{~~in~~}M$$
and
$$F(x,0)=v(x,0)-\epsilon=\beta(0) h^{-1}h_{g_0}-\lambda(0)-\epsilon<C_0\mbox{~~on~~}\partial M.$$
We assume that $t_0$ is the smallest time satisfying \eqref{4.10}, i.e.
\begin{equation}\label{4.11}
F(x,t)\leq C_0\mbox{ for all }(x,t)\in \overline{M}\times [0,t_0),
F(x,t_0)\leq C_0\mbox{ for all }x\in \overline{M}\mbox{ and }F(x_0,t_0)=C_0.
\end{equation}
This implies that
\begin{equation}\label{4.12}
v(x_0,t_0)=\max_{x\in\overline{M}}v(x,t_0)=C_0+\epsilon(t_0+1)\geq 0.
\end{equation}
We have the following two cases.\\
Case (i). If $x_0\in M$, then at $(x_0,t_0)$ we have
\begin{align*}
0\leq \partial_t F=&-\epsilon+\partial_t v\\
\leq&-\epsilon-(n-1)\alpha f^{-1}\Delta_g v-v^2
-\lambda v+(v+\lambda)\Big(\int_M -fd\mu_g\Big)^{-1}\Big(\int_M -fvd\mu_g\Big)\\
\leq&-\epsilon-(v+\lambda)\left[v-\Big(\int_M -fd\mu_g\Big)^{-1}\Big(\int_M -fvd\mu_g\Big)\right]\leq -\epsilon,
\end{align*}
where the first inequality follows from \eqref{4.11}, the second inequality follows from \eqref{v_t_1},
the third inequality follows from \eqref{4.12} and the fact that $f<0$ and $\alpha>0$ by \eqref{alpha_beta_lambda_bdd},
and the last  inequality follows from \eqref{4.12}. This contradicts $\epsilon>0$.\\
Case (ii). If $x_0\in\partial M$, then at $(x_0,t_0)$ we have
\begin{align*}
0\leq \partial_t F=&-\epsilon+\partial_t v\\
\leq&-\epsilon+
\frac{1}{2}\beta h^{-1}\frac{\partial}{\partial \nu_g}v-\frac{1}{2}v^2-\frac{1}{2}\lambda v+\frac{1}{2}(v+\lambda)\Big(\int_{\partial M}-hd\sigma_g\Big)^{-1}\Big(\int_{\partial M}-h vd\sigma_g\Big)\\
\leq &-\epsilon-\frac{1}{2}(v+\lambda)\left[v-\Big(\int_{\partial M}-hd\sigma_g\Big)^{-1}\Big(\int_{\partial M}-h vd\sigma_g\Big)\right]\leq -\epsilon,
\end{align*}
where the first inequality follows from \eqref{4.11}, the second inequality follows from \eqref{v_t_2},
the third inequality follows from \eqref{4.12} and the fact that $h<0$ and $\beta>0$ by \eqref{alpha_beta_lambda_bdd},
and the last  inequality follows from \eqref{4.12}. This again contradicts $\epsilon>0$.

This proves the claim. That is, $F=v-\epsilon(t+1)<C_0$ on $M\times [0,\infty)$.
Now the assertion follows by letting $\epsilon \to 0$.
\end{proof}

For $p\geq 2$, we let
$$F_p(t)=
-\alpha^{-1}\int_Mf|\alpha f^{-1}R_g-\lambda|^pd\mu_{g}-2(n-1)\beta^{-1}\int_{\partial M}h|\beta h^{-1}h_g-\lambda|^pd\sigma_{g}.$$
It follows from  \eqref{lambda_t} that
\begin{equation}\label{lambda_t_bdd}
\partial_t\lambda=\frac{n-2}{2} F_2(t).
\end{equation}
By \eqref{alpha_beta_lambda_bdd}, Lemma \ref{volume_bdd} and the assumption that $f,h<0$,
we can apply H\"{o}lder's inequality to \eqref{alpha_t} and \eqref{beta_t}
to get
\begin{equation}\label{alpha_t_beta_t_bdd}
|\alpha_t|\leq C F_2(t)^{\frac{1}{2}}\hspace{2mm}\mbox{ and }\hspace{2mm}|\beta_t|\leq C F_2(t)^{\frac{1}{2}}
\end{equation}
for some uniform constant $C$.

By using \eqref{v}, we can rewrite $F_p(t)$ in terms of $v$ as
$$F_p(t)=
-\alpha^{-1}\int_Mf|v|^pd\mu_{g}-2(n-1)\beta^{-1}\int_{\partial M}h|v|^pd\sigma_{g}.$$
Then we have
\begin{align*}
\frac{d}{dt}F_p(t)
&=-p\alpha^{-1}\int_Mf|v|^{p-2}v\partial_tvd\mu_{g}-2(n-1)p\beta^{-1}\int_{\partial M}h|v|^{p-2}v\partial_tvd\sigma_{g}\\
&\hspace{4mm}-\frac{n}{2}\alpha^{-1}\int_Mf|v|^p vd\mu_{g}-(n-1)^2\beta^{-1}\int_{\partial M}h|v|^p vd\sigma_{g}\\
&\hspace{4mm}+\alpha^{-2}\partial_t\alpha\int_Mf|v|^pd\mu_{g}+2(n-1)\beta^{-2}\partial_t\beta\int_{\partial M}h|v|^pd\sigma_{g}.
\end{align*}
Combining this with  \eqref{v_t_1} and \eqref{v_t_2},
we obtain
\begin{align}\label{F_p_t}
&\frac{d}{dt}F_p(t)\no\\
=&-(n-1)p(p-1)\int_M|v|^{p-2}|\nabla v|_g^2d\mu_g\no\\
&+\left[\Big(\frac{p\lambda}{\alpha}-\frac{(p-1)\partial_t\alpha}{\alpha^2}\Big)\int_Mf|v|^{p}d\mu_{g}+(n-1)\Big(\frac{p\lambda}{\beta}-\frac{2(p-1)\partial_t \beta}{\beta^2}\Big)\int_{\partial M}h|v|^{p}d\sigma_{g}\right]\no\\
&+\left[p\Big(\frac{\partial_t\lambda}{\alpha}-\frac{\lambda\partial_t\alpha}{\alpha^2}\Big)\int_Mf|v|^{p-2}vd\mu_{g}+2(n-1)p\Big(\frac{\partial_t\lambda}{\beta}-\frac{\lambda\partial_t\beta}{\beta^2}\Big)\int_{\partial M}h|v|^{p-2}vd\sigma_{g}\right]\no\\
&+\left[(p-\frac{n}{2})\alpha^{-1}\int_Mf |v|^p vd\mu_{g}+(n-1)(p-n+1)\beta^{-1}\int_{\partial M}h|v|^p vd\sigma_{g}\right]\\
:=&I^{(p)}+II^{(p)}+III^{(p)}+IV^{(p)},\no
\end{align}
where we have used the following identity which follows from integration by parts:
\begin{equation*}
\int_M|v|^{p-2}v\Delta_g vd\mu_{g}-\int_{\partial M}|v|^{p-2}v\frac{\partial}{\partial\nu_g}vd\sigma_{g}
=-(p-1)\int_M|v|^{p-2}|\nabla v|_g^2d\mu_g.
\end{equation*}

In the following, we first recall the Sobolev inequality and sharp Sobolev trace inequality in a compact manifold with boundary. Let $\omega_n$ denote the volume of the unit sphere in $\mathbb{R}^{n+1}$. Then it follows from \cite[Theorem 2.30]{aubin_book} or \cite[Theorem 2]{Hebey-Vaugon} that given any $\epsilon>0$, there exists $A_\epsilon=A(\epsilon,M,g_0)>0$ such that
\begin{equation}\label{Sobolev_ineq}
\left(\int_M |\varphi|^\criti d\mu_{g_0}\right)^{\frac{n-2}{n}}\leq (K(n,2)+\epsilon)\int_M |\nabla \varphi|_{g_0}^2 d\mu_{g_0}+A_\epsilon \int_M \varphi^2 d\mu_{g_0}
\end{equation}
for all $\varphi \in H^1(M,g_0)$, where $K(n,2)=\frac{4}{n(n-2)}\omega_n^{-\frac{2}{n}}$. The sharp Sobolev trace inequality was proved by Li and Zhu (cf. \cite[Theorem 0.1]{li-zhu1}): There exists $B=B(M,g_0)>0$ such that
\begin{equation}\label{Sobolev_trace_ineq}
\left(\int_{\partial M}|\varphi|^{\frac{2(n-1)}{n-2}}d\sigma_{g_0}\right)^{\frac{n-2}{n-1}}\leq S\int_M|\nabla \varphi|_{g_0}^2d\mu_{g_0}
+B\int_{\partial M}\varphi^2d\sigma_{g_0}
\end{equation}
for all $\varphi\in H^1(M,g_0)$, where $S=\frac{2}{n-2}\omega_{n-1}^{-1}$.

\begin{lemma}\label{lem4.2}
For any fixed $T>0$, there exists a positive constant $C$ independent of $T$ such that
\begin{equation}\label{est:F_2_T}
F_2(t)\leq C \quad\hbox{~~for all~~} 0 \leq t \leq T.
\end{equation}
Moreover, if the flows exist for all time, there holds
$$\lim_{t\to\infty}F_2(t)=0.$$
\end{lemma}
\begin{proof}
Notice that
\begin{equation*}
I^{(2)}=-2(n-1)\int_M|\nabla v|_g^2d\mu_g.
\end{equation*}
By using  \eqref{alpha_beta_lambda_bdd},  \eqref{alpha_t}, \eqref{beta_t}, \eqref{v} and Lemma \ref{lem:v_bdd}, we can estimate
\begin{equation}\label{alpha_t_upper_bdd}
\partial_t\alpha \leq C_0 \alpha \leq C_0\alpha_2
\end{equation}
and
\begin{equation}\label{beta_t_upper_bdd}
\partial_t\beta \leq \frac{\beta C_0}{2}\leq \frac{\beta_2 C_0}{2}.
\end{equation}
From \eqref{alpha_beta_lambda_bdd}, \eqref{alpha_t_upper_bdd} and \eqref{beta_t_upper_bdd}, we can estimate
\begin{equation*}
\begin{split}
II^{(2)}&=(2\alpha^{-1}\lambda-\alpha^{-2}\partial_t\alpha)\int_Mfv^{2}d\mu_{g}
+(n-1)(2\beta^{-1}\lambda-2\beta^{-2}\partial_t\beta)\int_{\partial M}hv^{2}d\sigma_{g}\\
&\leq CF_2(t).
\end{split}
\end{equation*}
We rearrange
\begin{align*}
III^{(2)}=&2\partial_t\lambda\left[\alpha^{-1}\int_Mfvd\mu_{g}+2(n-1)\beta^{-1}\int_{\partial M}hvd\sigma_{g}\right]\\
&-2\lambda\alpha^{-2}\partial_t\alpha\int_Mfvd\mu_{g}
-4(n-1)\lambda\beta^{-2}\partial_t\beta\int_{\partial M}hvd\sigma_{g}.
\end{align*}
It follows from \eqref{alpha_t} and \eqref{beta_t} that
$$\partial_t\alpha=-a^{-\frac{n}{2}}\alpha^{\frac{2-n}{2}}\int_M fvd\mu_g \hbox{~~and~~} \partial_t\beta=-\frac{1}{2}b^{1-n}\beta^{2-n}\int_{\partial M}h d\sigma_g.$$
From this, we obtain
\begin{equation*}
\alpha^{-1}\int_M fvd\mu_g=-\frac{2}{n-2}\partial_t\left[a\Big(\int_M -fd\mu_g\Big)^{\frac{n-2}{n}}\right]
\end{equation*}
and
\begin{equation*}
\beta^{-1}\int_{\partial M}hd\sigma_g=-\frac{2}{n-2}\partial_t\left[b\Big(\int_{\partial M}-hd\sigma_g\Big)^{\frac{n-2}{n-1}}\right].
\end{equation*}
Thus, we get
\begin{align*}
&\alpha^{-1}\int_Mfvd\mu_{g}+2(n-1)\beta^{-1}\int_{\partial M}hvd\sigma_{g}\\
=&-\frac{2}{n-2}\partial_t\left[a\Big(\int_M -fd\mu_g\Big)^{\frac{n-2}{n}}+2(n-1)b\Big(\int_{\partial M}-hd\sigma_g\Big)^{\frac{n-2}{n-1}}\right]=0,
\end{align*}
where the last equality follows from the normalization \eqref{conserved_quantity}.
Therefore,  by \eqref{alpha_beta_lambda_bdd},  \eqref{lambda_t_bdd}, \eqref{alpha_t_beta_t_bdd} and H\"{o}lder's inequality, we obtain
\begin{equation*}
III^{(2)}=-2\lambda\alpha^{-2}\partial_t\alpha\int_Mfvd\mu_{g}
-4(n-1)\lambda\beta^{-2}\partial_t\beta\int_{\partial M}hvd\sigma_{g}\leq CF_2(t).
\end{equation*}
It remains to estimate the last term:
\begin{equation*}
IV^{(2)}=(2-\frac{n}{2})\alpha^{-1}\int_Mf v^{3}d\mu_{g}+
(n-1)(3-n)\beta^{-1}\int_{\partial M}hv^{3}d\sigma_{g}.
\end{equation*}
We divide it into two cases:

\noindent Case (i). If $n \geq 4$, it follows
from Lemma \ref{lem:v_bdd} and the assumptions $f<0$ and $h<0$ that
\begin{equation*}
IV^{(2)}\leq C_1\left[(2-\frac{n}{2})\alpha^{-1}\int_Mf v^2d\mu_{g}+
(n-1)(3-n)\beta^{-1}\int_{\partial M}hv^2d\sigma_{g}\right]\leq CF_2(t).
\end{equation*}
Putting these facts together, we have
\begin{equation}\label{4.17}
\frac{d}{dt}F_2(t)+2(n-1)\int_M|\nabla v|_g^2d\mu_g\leq CF_2(t).
\end{equation}
 By \eqref{energy_gap} and Lemma \ref{volume_bdd}, there exists a constant $C>0$ independent of $T$ such that
 $$\int_0^T F_2(t)dt\leq C.$$
 Then estimate \eqref{est:F_2_T} follows from integration over $(0,T)$. If the flows globally exist,  we obtain
\begin{equation}\label{4.18}
\int_0^\infty F_2(t)dt<\infty.
\end{equation}
Then there exists a sequence $\{t_j;j \in \mathbb{N}\}$ with $t_j\to\infty$ as $j\to\infty$ such that
\begin{equation}\label{4.20}
F_2(t_j)\to 0\mbox{ as }j\to\infty.
\end{equation}
Integrating \eqref{4.17} over $(t_j,t)$ to show
$$F_2(t)\leq F_2(t_j)+C\int_{t_j}^\infty F_2(s)ds.$$
By letting $j \to \infty$, we obtain $F_2(t) \to 0$ as $t \to \infty$.

\noindent Case (ii). When $n=3$, we again rewrite \eqref{F_p_t} as
\begin{align*}
\frac{d}{dt}F_2(t)&=-4\int_M |\nabla v|_g^2 d\mu_g+II+III+\frac{1}{2}\alpha^{-1}\int_Mf v^3d\mu_{g}\no\\
&\leq -C_1 \int_M |\nabla v|_{g_0}^2d\mu_{g_0}+C\int_M |v|^3 d\mu_{g_0}+CF_2(t)
\end{align*}
for some positive uniform constants $C_1$ and $C$, where the last inequality follows from the uniform bounds of $u$ by Lemmas \ref{lower_bdd_u} and \ref{upper_bdd_u}. By H\"older's and Young's inequalities, given any $\epsilon_0>0$, we obtain
\begin{equation}\label{int_est_v^3}
\int_M |v|^3 d\mu_{g_0}\leq \epsilon_0\left( \int_M |v|^6d\mu_{g_0}\right)^{\frac{1}{3}}+C_{\epsilon_0}\left(\int_M v^2 d\mu_{g_0}\right)^3.
\end{equation}
For sufficiently small $\epsilon_0>0$, from \eqref{int_est_v^3} and Lemmas \ref{lower_bdd_u}, \ref{upper_bdd_u}, we apply \eqref{Sobolev_ineq} to $v$ to show
\begin{equation}\label{ineq:F_2_t}
(1+F_2(t)^2)^{-1}\left(\frac{d}{dt}F_2(t)+\int_M |\nabla v|_g^2d\mu_g\right)\leq C F_2(t).
\end{equation}
Using the same argument in case (i), integrating \eqref{ineq:F_2_t} over $(0,T)$ gives estimate \eqref{est:F_2_T}. If the flows globally exist, by using the sequence $\{t_j\}$ as above, integrating \eqref{ineq:F_2_t} over $(t_j,t)$ to get
$$\arctan{F_2(t)}\leq \arctan{F_2(t_j)}+ C\int_{t_j}^\infty F_2(t) dt.$$
Letting $j \to \infty$ in the above inequality and from \eqref{4.18}, \eqref{4.20}, we obtain $\lim_{t \to \infty}F_2(t)=0$ as required.
\end{proof}

For simplicity, in the following we will give the $L^p$-estimates for $R_{g(t)}$ and $h_{g(t)}$ in infinite interval $[0,\infty)$ directly, since their estimates on finite time interval $[0,T]$ can be obtained by mimicking the proof of Lemma \ref{lem4.2}.

\begin{lemma}\label{mid_est_lem}
For $n \geq 4$ and $2 \leq p \leq \frac{n}{2}$, there hold
\begin{equation*}
\int_0^\infty F_p(t)dt<\infty
\quad\hbox{and}\quad\lim_{t \to \infty}F_p(t)=0.
\end{equation*}
\end{lemma}
\begin{proof}
We prove this lemma by induction on $p$. From  \eqref{4.18} and Lemma \ref{lem4.2}, Lemma \ref{mid_est_lem} holds for $p=2$. Suppose Lemma \ref{mid_est_lem} holds for some $p\in[2,\frac{n}{2}]$.
We can estimate some of the terms on the right hand side of \eqref{F_p_t} as follows:
by  \eqref{alpha_beta_lambda_bdd}, \eqref{alpha_t_beta_t_bdd} and Lemma \ref{lem4.2}, we have
\begin{align}\label{II}
&|II^{(p)}|\no\\
=&\left|\Big(\frac{p\lambda}{\alpha}-\frac{(p-1)\partial_t\alpha}{\alpha^2}\Big)\int_Mf|v|^{p}d\mu_{g}+(n-1)\Big(\frac{p\lambda}{\beta}-\frac{2(p-1)\partial_t \beta}{\beta^2}\Big)\int_{\partial M}h|v|^{p}d\sigma_{g}\right|\no\\
\leq& C(1+F_2(t)^{\frac{1}{2}})F_p(t)\leq CF_p(t).
\end{align}
Also, we have
\begin{align}\label{III}
&|III^{(p)}|\no\\
=&\left|p\Big(\frac{\partial_t\lambda}{\alpha}-\frac{\lambda\partial_t\alpha}{\alpha^2}\Big)\int_Mf|v|^{p-2}vd\mu_{g}+2(n-1)p\Big(\frac{\partial_t\lambda}{\beta}-\frac{\lambda\partial_t\beta}{\beta^2}\Big)\int_{\partial M}h|v|^{p-2}vd\sigma_{g}\right|\no\\
\leq& C(F_2(t)+F_2(t)^{\frac{1}{2}})F_{p-1}(t)\leq CF_{p}(t),
\end{align}
where the first inequality follows from \eqref{alpha_beta_lambda_bdd} and \eqref{lambda_t_bdd},
and the second inequality follows from H\"{o}lder's inequality, Lemmas \ref{volume_bdd} and \ref{lem4.2}.
Therefore, we can  rewrite \eqref{F_p_t} as
\begin{align*}
&(\frac{n}{2}-p)\alpha^{-1}\int_Mf |v|^p vd\mu_{g}+(n-1)(n-1-p)\beta^{-1}\int_{\partial M}h|v|^p vd\sigma_{g}\\
=& -\frac{d}{dt}F_p+I^{(p)}+II^{(p)}+III^{(p)}\\
\leq& -\frac{d}{dt}F_p+CF_p,
\end{align*}
where the last inequality follows from \eqref{II} and \eqref{III}. Together with the induction assumption for $p$, integrating the above inequality over $(0,\infty)$ indicates that
$$\int_0^\infty\int_M f|v|^p v d\mu_g dt<\infty \hbox{~~and~~} \int_0^\infty \int_{\partial M}h |v|^p v d\sigma_g dt<\infty.$$
Notice that
\begin{align*}
-\int_M f |v|^{p+1}d\mu_g=&-2\int_{M \cap \{v>0\}}f|v|^p v d\mu_g+\int_M f|v|^p v d\mu_g\\
\leq& C\int_M |v|^p d\mu_g+\int_M f|v|^p v d\mu_g,
\end{align*}
where the last inequality follows from Lemma \ref{lem:v_bdd} and the fact that $f<0$. Thus we obtain
\begin{equation}\label{est:int_M_F_p+1}
-\int_0^\infty\int_M f|v|^{p+1}d\mu_g dt<\infty.
\end{equation}
Similarly, we also have
\begin{equation}\label{est:int_bdry_F_p+1}
-\int_0^\infty\int_{\partial M} h |v|^{p+1}d\sigma_gdt<\infty.
\end{equation}
Thus \eqref{est:int_M_F_p+1} and \eqref{est:int_bdry_F_p+1} imply that $\int_0^\infty F_{p+1}(t)dt<\infty$.

If $p+1<\frac{n}{2}$, going back to \eqref{F_p_t} for $p+1$, we obtain
$$\frac{d}{dt}F_{p+1}\leq CF_{p+1}.$$
Integrating the above inequality over $(0,\infty)$ and using \eqref{est:int_M_F_p+1} and \eqref{est:int_bdry_F_p+1}, we obtain
$$\lim_{t \to \infty}F_{p+1}(t)=0.$$
Then we can repeat the above steps until reaching some $p$ with $p+1>\frac{n}{2}$ and at the same time $p\leq \frac{n}{2}$. By H\"{o}lder's and Young's inequalities, we have
\[F_{\frac{n}{2}}\leq F_p^sF_{p+1}^{1-s}\leq sF_p+(1-s)F_{p+1},\]
where $s=p+1-\frac{n}{2}\in(0,1)$ when $p\neq \frac{n}{2}$. Thus $\int_0^\infty F_{\frac{n}{2}}(t)dt<\infty$ in both cases and the induction is finished. Moreover, returning to \eqref{F_p_t} for $p=\frac{n}{2}$, we conclude that $\lim_{t \to \infty}F_{\frac{n}{2}}(t)=0.$
This completes the proof.
\end{proof}

\begin{lemma}\label{second_mid_est}
For $\frac{n}{2}\leq p \leq n-1$, there hold
\begin{equation}\label{second_mid_est:int_F_p}
\int_0^\infty F_p(t)dt<\infty
\end{equation}
and
\begin{equation}\label{second_mid_est:F_p}
\lim_{t \to \infty}F_p(t)=0.
\end{equation}
\end{lemma}
\begin{proof}
As the proof of Lemma \ref{mid_est_lem}, we prove it by induction
on $p$. It follows from Lemma \ref{mid_est_lem} that \eqref{second_mid_est:int_F_p} and \eqref{second_mid_est:F_p}
are true for $p=\frac{n}{2}$.
We now suppose that \eqref{second_mid_est:int_F_p} and \eqref{second_mid_est:F_p} are true for some
$\frac{n}{2}\leq p<n-1$.
We rewrite \eqref{F_p_t} as
\begin{align}\label{A}
&(n-1)(n-1-p)\beta^{-1}\int_{\partial M}h|v|^p vd\sigma_{g}\no\\
=& -\frac{d}{dt}F_p+I^{(p)}+II^{(p)}+III^{(p)}
+(p-\frac{n}{2})\alpha^{-1}\int_Mf |v|^p vd\mu_{g}\no\\
\leq& -\frac{d}{dt}F_p+CF_p+I^{(p)}+(p-\frac{n}{2})\alpha^{-1}\int_Mf |v|^p vd\mu_{g},
\end{align}
where the last inequality follows from \eqref{II} and \eqref{III}. Recall that
$$
I^{(p)}=-(n-1)p(p-1)\int_M|v|^{p-2}|\nabla v|_g^2d\mu_g.$$

Notice that if $p>\frac{n}{2}$, then  for any $\epsilon>0$, there exists a constant $C(\epsilon)$ such that
\begin{align}\label{a}
\left|\int_{M}f |v|^p vd\mu_{g}\right|
&\leq C\int_{M} |v|^{p+1}d\mu_{g}\no\\
&\leq C\left(\int_{M} |v|^{\frac{np}{n-2}}d\mu_{g}\right)^{\frac{n-2}{2p}}\left(\int_{M} |v|^{p}d\mu_{g}\right)^{\frac{2p-n+2}{2p}}\no\\
&\leq \epsilon\left(\int_{M} |v|^{\frac{np}{n-2}}d\mu_{g}\right)^{\frac{n-2}{n}}
+C(\epsilon)\left(\int_{M} |v|^{p}d\mu_{g}\right)^{\frac{2p-n+2}{2p-n}},
\end{align}
where we have used the assumption $p>\frac{n}{2}$ in applying H\"{o}lder's and Young's inequalities.
Since $u$ is uniformly bounded by Lemmas \ref{lower_bdd_u} and \ref{upper_bdd_u},
we can deduce from \eqref{Sobolev_ineq} that
\begin{align}\label{c}
\left(\int_{M} |v|^{\frac{np}{n-2}}d\mu_{g}\right)^{\frac{n-2}{n}}
&\leq C\left(\int_{M} |v|^{\frac{np}{n-2}}d\mu_{g_0}\right)^{\frac{n-2}{n}}\no\\
&\leq C\int_M |v|^{p-2}|\nabla v|_{g_0}^2 d\mu_{g_0}+C\int_M |v|^{p} d\mu_{g_0}\no\\
&\leq C\int_M |v|^{p-2}|\nabla v|_{g}^2 d\mu_{g}+C\int_M |v|^{p} d\mu_{g}.
\end{align}
Combining \eqref{a} and \eqref{c}, we get
\begin{align}\label{e}
\left|(p-\frac{n}{2})\alpha^{-1}\int_{M}f |v|^p vd\mu_{g}\right|
\leq& C\int_{M} |v|^{p+1}d\mu_{g}\no\\
\leq& C\epsilon\int_M |v|^{p-2}|\nabla v|_{g}^2 d\mu_{g}
+CF_p(t)+CF_p(t)^{\frac{2p-n+2}{2p-n}}\no\\
\leq& C\epsilon\int_M |v|^{p-2}|\nabla v|_{g}^2 d\mu_{g}
+CF_p(t),
\end{align}
where we have used \eqref{alpha_beta_lambda_bdd} and the induction assumption \eqref{second_mid_est:F_p} for $p$.
Combining \eqref{A} and \eqref{e} and choosing $\epsilon$ small enough, we have
\begin{equation*}
\begin{split}
(n-1)(n-1-p)\beta^{-1}\int_{\partial M}h|v|^p vd\sigma_{g}+C\int_M |v|^{p-2}|\nabla v|_{g}^2 d\mu_{g}
\leq& -\frac{d}{dt}F_p(t)+CF_p(t).
\end{split}
\end{equation*}
Together with \eqref{second_mid_est:int_F_p} and \eqref{second_mid_est:F_p}, integrating the above inequality over $(0,\infty)$ indicates that
\begin{equation}\label{est:int_dv}
\int_0^\infty \int_M|v|^{p-2}|\nabla v|_g^2d\mu_gdt<\infty
\end{equation}
and
\begin{equation}\label{est:int_bdry_|v|^pv}
\int_0^\infty \int_{\partial M}h |v|^p v d\sigma_g dt<\infty.
\end{equation}
Then it follows from \eqref{e} and \eqref{est:int_dv} that
\begin{equation}\label{est:int_M_|v|^p+1}
\int_0^\infty\int_M |v|^{p+1}d\mu_gdt<\infty.
\end{equation}
Notice that
\begin{align*}
-\int_{\partial M}h |v|^{p+1} d\sigma_g dt=&-2\int_{\partial M\cap\{v>0\}}h |v|^p v d\sigma_g+\int_{\partial M}h |v|^p v d\sigma_g\\
\leq& C\int_{\partial M}|v|^p d\sigma_g+\int_{\partial M}h |v|^p v d\sigma_g,
\end{align*}
where the last inequality follows from Lemma \ref{lem:v_bdd} and the fact that $h<0$. Thus we obtain from \eqref{est:int_bdry_|v|^pv} that
\begin{equation}\label{second_est:int_bdry_F_p+1}
-\int_0^\infty\int_{\partial M}h |v|^{p+1}d\sigma_g dt<\infty.
\end{equation}
Combing this with \eqref{est:int_M_|v|^p+1}, we prove \eqref{second_mid_est:F_p} for $p+1$.

If $p+1<n-1$, we can go back to \eqref{F_p_t} for $p+1$ to obtain
$$\frac{d}{dt}F_{p+1}(t)\leq CF_{p+1}(t).$$
By integrating the above inequality over $(0,\infty)$
and using \eqref{second_mid_est:F_p} for $p+1$, we obtain
$$\lim_{t \to \infty}F_{p+1}(t)=0.$$

Therefore, we can repeat the above steps until we reach some $p$ satisfying $n-2<p\leq n-1$. By H\"{o}lder's and Young's inequalities, we have
\[F_{n-1}(t)\leq F_p(t)^sF_{p+1}(t)^{1-s}\leq sF_p(t)+(1-s)F_{p+1}(t),\]
where $s=p+2-n\in(0,1)$ when $p\neq n-1$. Thus $\int_0^\infty F_{n-1}(t)dt<\infty$ in both cases and the induction is finished. Moreover, returning to \eqref{F_p_t} for $p=n-1$, we conclude that $\lim_{t \to \infty}F_{n-1}(t)=0$. This completes the proof.
\end{proof}

\begin{lemma}\label{lem4.3}
For any $p>n-1$, there holds
\begin{equation*}
\frac{d}{dt}F_p(t)+CF_{\frac{(n-1)p}{n-2}}(t)^{\frac{n-2}{n-1}}
\leq CF_p(t)+CF_p(t)^{\frac{p-n+2}{p-n+1}}.
\end{equation*}
\end{lemma}
\begin{proof}
Again we are going to use \eqref{F_p_t} and estimate the terms on its right hand side.
Note that
$$
I^{(p)}=-(n-1)p(p-1)\int_M|v|^{p-2}|\nabla v|_g^2d\mu_g
\leq -C\int_M|v|^{p-2}|\nabla v|_{g_0}^2d\mu_{g_0}
$$
by Lemmas \ref{lower_bdd_u}, \ref{upper_bdd_u},
where $C$ is a positive uniform constant. Applying \eqref{Sobolev_ineq} and \eqref{Sobolev_trace_ineq} to $\varphi=|v|^{\frac{p}{2}}$, we obtain
\begin{equation}\label{I}
I^{(p)}\leq-C\left(\int_{M}|v|^{\frac{np}{n-2}}d\mu_{g}\right)^{\frac{n-2}{n}}-C\left(\int_{\partial M}|v|^{\frac{(n-1)p}{n-2}}d\sigma_{g}\right)^{\frac{n-2}{n-1}}+F_p(t).
\end{equation}

It follows from H\"{o}lder's and Young's inequalities
that for any $\epsilon>0$, there exists a constant $C(\epsilon)$ such that
\begin{equation}\label{b}
\begin{split}
\left|\int_{\partial M}h|v|^p vd\sigma_{g}\right|
&\leq C\int_{\partial M}|v|^{p+1}d\sigma_{g}\\
&\leq C\left(\int_{\partial M}|v|^{\frac{(n-1)p}{n-2}}d\sigma_{g}\right)^{\frac{n-2}{p}}
\left(\int_{\partial M}|v|^pd\sigma_{g}\right)^{\frac{p-n+2}{p}}\\
&\leq \epsilon\left(\int_{\partial M}|v|^{\frac{(n-1)p}{n-2}}d\sigma_{g}\right)^{\frac{n-2}{n-1}}
+C(\epsilon)\left(\int_{\partial M}|v|^pd\sigma_{g}\right)^{\frac{p-n+2}{p-n+1}},
\end{split}
\end{equation}
where we have used the assumption $p>n-1$ in applying H\"{o}lder's and Young's inequalities.
Since $u$ is uniformly bounded by Lemmas \ref{lower_bdd_u} and \ref{upper_bdd_u},
we can deduce from  \eqref{Sobolev_trace_ineq} that
\begin{equation}\label{d}
\begin{split}
\left(\int_{\partial M}|v|^{\frac{(n-1)p}{n-2}}d\sigma_{g}\right)^{\frac{n-2}{n-1}}
&\leq C\left(\int_{\partial M}|v|^{\frac{(n-1)p}{n-2}}d\sigma_{g_0}\right)^{\frac{n-2}{n-1}}\\
&\leq C\int_M|v|^{p-2}|\nabla v|_{g_0}^2d\mu_{g_0}
+C\int_{\partial M}|v|^{p}d\sigma_{g_0}\\
&\leq C\int_M|v|^{p-2}|\nabla v|_{g}^2d\mu_{g}
+C\int_{\partial M}|v|^pd\sigma_{g}.
\end{split}
\end{equation}
Therefore, by \eqref{a}, \eqref{c}, \eqref{b} and \eqref{d}, we can estimate
\begin{align}\label{IV}
&|IV^{(p)}|\no\\
=&\left|(p-\frac{n}{2})\alpha^{-1}\int_{M}f |v|^p vd\mu_{g}+(n-1)(p-n+1)\beta^{-1}\int_{\partial M}h|v|^p vd\sigma_{g}\right|\no\\
\leq& C\left|\int_{M}f |v|^p vd\mu_{g}\right|+C\left|\int_{\partial M}h|v|^p vd\sigma_{g}\right|\no\\
\leq& C\epsilon\left(\int_{M}|v|^{\frac{np}{n-2}}d\mu_{g}\right)^{\frac{n-2}{n}}
+C\epsilon\left(\int_{\partial M}|v|^{\frac{(n-1)p}{n-2}}d\sigma_{g}\right)^{\frac{n-2}{n-1}}+CF_p(t)+CF_p(t)^{\frac{p-n+2}{p-n+1}},
\end{align}
where we have used \eqref{alpha_beta_lambda_bdd} and the fact that
$$F_p(t)^{\frac{2p-n+2}{2p-n}}\leq F_p(t)+F_p(t)^{\frac{p-n+2}{p-n+1}}.$$

Combining \eqref{II}, \eqref{III}, \eqref{I}, \eqref{IV} and choosing $\epsilon$ small enough, we obtain
\begin{align}\label{4.26}
\frac{d}{dt}F_p(t)+C\left(\int_M|v|^{\frac{np}{n-2}}d\mu_{g}\right)^{\frac{n-2}{n}}
+C\left(\int_{\partial M}|v|^{\frac{(n-1)p}{n-2}}d\sigma_{g}\right)^{\frac{n-2}{n-1}}
\leq CF_p(t)+CF_p(t)^{\frac{p-n+2}{p-n+1}}.
\end{align}
By H\"{o}lder's inequality and  Lemma \ref{volume_bdd}, we have
\begin{equation}\label{4.27}
\left(\int_M|v|^{\frac{(n-1)p}{n-2}}d\mu_{g}\right)^{\frac{n-2}{n-1}}
\leq\left(\int_Md\mu_g\right)^{\frac{n-2}{n(n-1)}}
\left(\int_M|v|^{\frac{np}{n-2}}d\mu_{g}\right)^{\frac{n-2}{n}}
\leq C\left(\int_M|v|^{\frac{np}{n-2}}d\mu_{g}\right)^{\frac{n-2}{n}}.
\end{equation}
Also, it follows from \eqref{alpha_beta_lambda_bdd} that
\begin{align}\label{4.28}
F_{\frac{(n-1)p}{n-2}}(t)^{\frac{n-2}{n-1}}
&\leq C\left(\int_M|v|^{\frac{(n-1)p}{n-2}}d\mu_{g}+\int_{\partial M}|v|^{\frac{(n-1)p}{n-2}}d\sigma_{g}\right)^{\frac{n-2}{n-1}}\no\\
&\leq C\left(\int_M|v|^{\frac{(n-1)p}{n-2}}d\mu_{g}\right)^{\frac{n-2}{n-1}}+C\left(\int_{\partial M}|v|^{\frac{(n-1)p}{n-2}}d\sigma_{g}\right)^{\frac{n-2}{n-1}}.
\end{align}
Now the assertion follows from \eqref{4.26}-\eqref{4.28}.
\end{proof}

\begin{proposition}\label{prop4.1}
For any $n-1\leq p<\infty$, there holds
$$F_p(t)\to 0\mbox{ as }t\to\infty.$$
\end{proposition}
\begin{proof}
Set $p_k=(n-1)\displaystyle\left(\frac{n-1}{n-2}\right)^k$, $k\in \mathbb{N}$.
Assume that for $p=p_k$ and for some  $\delta_k\in(0, 1],$
there holds
\begin{equation}\label{induction_k}
\int_0^\infty F_{p_k}(t)^{\delta_k}dt<\infty.
\end{equation}
Note that \eqref{induction_k} for $k=0$ is true
with $\delta_0=1$, thanks to Lemma \ref{second_mid_est}.
We are going to show that
\begin{equation}\label{assertion}
\lim_{t\to\infty} F_{p_k}(t)=0
\end{equation}
and to establish \eqref{induction_k} with $p=p_{k+1}$
with $\delta_{k+1}=\displaystyle\frac{n-2}{n-1}$ for all $k\in\mathbb{N}$.
It follows from Lemma \ref{lem4.3} that
\begin{equation}\label{4.33}
\frac{d}{dt}F_{p_k}(t)+CF_{p_{k+1}}(t)^{\delta_{k+1}}
\leq CF_{p_k}(t)+CF_{p_k}(t)^{\frac{p_k-n+2}{p_k-n+1}}.
\end{equation}
By \eqref{induction_k}, there exits a sequence $\{t_j;j \in \mathbb{N}\}$
with $t_j\to\infty$ as $j\to\infty$ such that
\begin{equation}\label{4.34}
F_{p_k}(t_j)\to 0\mbox{ as }j\to\infty.
\end{equation}
If we define
$$H(t)=\int_0^{F_{p_k}(t)}\frac{ds}{s^{\nu_1}+s^{\nu_2}},$$
where
$0\leq\nu_1=1-\delta_k<1$ and $\nu_2=\displaystyle\frac{p_k-n+2}{p_k-n+1}-\delta_k\geq 0$,
then it follows from \eqref{4.34} that
\begin{equation}\label{4.35}
H(t_j)=\int_0^{F_{p_k}(t_j)}\frac{ds}{s^{\nu_1}+s^{\nu_2}}\leq \int_0^{F_{p_k}(t_j)}\frac{ds}{s^{\nu_1}}=
\frac{1}{\delta_k}F_{p_k}(t_j)^{\delta_k}\to 0\mbox{ as }j\to\infty.
\end{equation}
Also, we have
$$\frac{d}{dt}H(t)\leq CF_{p_k}(t)^{\delta_k}$$
by \eqref{4.33} and the definition of $H(t)$.
Integrating the above differential inequality from $t_j$ to $t$, we obtain
\begin{equation*}
H(t)\leq H(t_j)+C\int_{t_j}^\infty F_{p_k}(t)^{\delta_k}dt.
\end{equation*}
Combining this with \eqref{induction_k} and \eqref{4.35}, we
have
\begin{equation}\label{4.36}
H(t)\to 0\quad \mbox{ as }t\to\infty.
\end{equation}
Now we claim that there exists a positive constant $C$ such that
\begin{equation}\label{4.37}
F_{p_k}(t)\leq C \quad \mbox{ for all }t\geq 0.
\end{equation}
Otherwise, there exists a sequence $\{t_j;j\in \mathbb{N}\}$
with $t_j\to\infty$ as $j\to\infty$ such that
$F_{p_k}(t_j)> 1$ for all $k\in\mathbb{N}$. However, we would have
$$H(t_j)=\int_0^{F_{p_k}(t_j)}\frac{ds}{s^{\nu_1}+s^{\nu_2}}\geq \int_0^{1}\frac{ds}{s^{\nu_1}+s^{\nu_2}}>0\mbox{ for all }k\in\mathbb{N},$$
which contradicts \eqref{4.36}.
Therefore,
$$H(t)=\int_0^{F_{p_k}(t)}\frac{ds}{s^{\nu_1}+s^{\nu_2}}\geq\frac{F_{p_k}(t)}{F_{p_k}(t)^{\nu_1}+F_{p_k}(t)^{\nu_2}}\geq CF_{p_k}(t)$$
since $\nu_1, \nu_2\geq 0$. This together with \eqref{4.36} imply that $F_{p_k}(t)\to 0$ as $t\to\infty$.

This proves \eqref{assertion}.
Moreover, integrating \eqref{4.33} over $(0,\infty)$
and using  \eqref{induction_k} and \eqref{assertion}, we obtain
$$\int_0^\infty F_{p_{k+1}}(t)^{\delta_{k+1}}dt<\infty.$$
That is, \eqref{induction_k} holds for $p=p_{k+1}$
with $\delta_{k+1}=\displaystyle\frac{n-2}{n-1}$.
This proves the assertion.
\end{proof}

It follows from Proposition \ref{prop4.1}
that $F_p(t)$ is bounded for any $p\geq 1$, i.e.
$$
-\alpha^{-1}\int_Mf|\alpha f^{-1}R_g-\lambda|^pd\mu_{g}-2(n-1)\beta^{-1}\int_{\partial M}h|\beta h^{-1}h_g-\lambda|^pd\sigma_{g}\leq C$$
for some positive constant $C$. Combining this with
\eqref{alpha_beta_lambda_bdd}, Lemmas \ref{lower_bdd_u}
and \ref{upper_bdd_u}, we obtain
\begin{equation}\label{curv_bdd}
\int_M|R_g|^pd\mu_{g_0}\leq C\mbox{ and }\int_{\partial M}|h_g|^pd\sigma_{g_0}\leq C
\end{equation}
for all $t\geq 0$.

Therefore, we have proved that the conformal factor $u$ is bounded below and above,
and that the scalar curvature $R_{g(t)}$ and the mean curvature $h_{g(t)}$ are $L^p(M,g_0)$ for all $p\geq 2$ in $\overline M \times [0,T]$. Following a standard argument in \cite[Proposition 2.6]{Brendle4}, we can show that $u(x,t)$ is H\"older continuous with respect to $x$ and $t$. By standard parabolic theory, all higher derivatives of $u$ are bounded in $\overline M \times [0,T]$. Then the long time existence of the flows \eqref{mixed_cur_flows} follows.

We next establish the asymptotic convergence of the flows by using a trick by Brendle \cite{Brendle1}. In particular, noticing that such bounds on $u,\|R_{g(t)}\|_{L^p(M,g_0)}$ and $\|h_{g(t)}\|_{L^p(M,g_0)}$ with all $p \geq 2$ are uniform in $\overline M \times [0,\infty)$ thanks to Lemmas \ref{lower_bdd_u}, \ref{upper_bdd_u} and \eqref{curv_bdd}, we conclude from the standard parabolic equation theory that all higher derivatives of $u$ are uniformly bounded. By Ascoli-Arzela theorem, there exists a sequence $\{t_k\}$ with $t_k \nearrow \infty$ and $u_\infty \in C^\infty(M)$ such that for all $m \in \mathbb{N}$, $\|u(t_k)-u_\infty\|_{C^m(M)}\to 0$ as $k \to \infty$. Since $E[u(t)]$ is uniformly bounded and non-increasing, we define
$$E_\infty=\lim_{t \to \infty}E[u(t)] \hbox{~~and then~~} \lambda_\infty =\lim_{t \to \infty}\lambda(t)=-\lim_{t \to \infty}E[u(t)]=-E_\infty.$$

We define
\begin{align*}
\mathcal{M}(u):=&\hbox{grad}\q(u)\\
=&\left\{L_{g_0}u-\alpha(u)^{-1}\lambda(u)fu^{\frac{n+2}{n-2}} \hbox{~~in~~} M; B_{g_0}u -\frac{n-2}{2}\beta(u)^{-1}\lambda(u)hu^{\frac{n}{n-2}}\hbox{~~on~~}\pa M\right\}
\end{align*}
and use the norm
\begin{align*}
&\|\mathcal{M}(u)\|\\
=&\|L_{g_0}u-\alpha(u)^{-1}\lambda(u)fu^{\frac{n+2}{n-2}}\|_{L^2(M,g_0)}+\|B_{g_0}u -\frac{n-2}{2}\beta(u)^{-1}\lambda(u)hu^{\frac{n}{n-2}}\|_{L^2(\pa M,g_0|_{\pa M})},
\end{align*}
where
\begin{equation}\label{eqn:constrant_time_funcs_converge}
\begin{split}
&\alpha(u)=\frac{1}{a}\Big(\int_M -fu^{\frac{2n}{n-2}}d\mu_{g_0}\Big)^{\frac{2}{n}}, \quad \beta(u)=\frac{1}{b}\Big(\int_{\partial M}-hu^{\frac{2(n-1)}{n-2}}d\sigma_{g_0}\Big)^{\frac{1}{n-1}},\\
&\lambda(u)=-\frac{E[u]}{a\Big(\int_M -f u^{\frac{2n}{n-2}}d\mu_{g_0}\Big)^{\frac{n-2}{n}}+2(n-1)b\Big(\int_{\partial M}-h u^{\frac{2(n-1)}{n-2}}d\sigma_{g_0}\Big)^{\frac{n-2}{n-1}}}.
\end{split}
\end{equation}
Since $\q[u]$ is real analytic, we employ Simon-Lojasiewicz inequality \cite[Theorem 3]{Simon} to show that there exists $\theta \in (0,\frac{1}{2})$ such that
$$\|\mathcal{M}(u(t))\|\geq C (E[u(t)]-E_\infty)^{1-\theta},$$
for all sufficiently large $t$, thanks to the normalization \eqref{conserved_quantity}. Since $u(t)$ is uniformly bounded, by \eqref{energy_rate} we obtain
$$CF_2(T)\geq (E[u(T)]-E_\infty)^{2(1-\theta)}\geq C\left(\int_T^\infty F_2(t)dt\right)^{2(1-\theta)}$$
for some sufficiently large $T$. By \cite[Lemma 4.1]{Huang-Takac}, we get
$$\int_0^\infty F_2(t)^{\frac{1}{2}}dt<\infty.$$
Given any initial datum $u_0$, the above estimate not only indicates the asymptotic uniqueness of the flows \eqref{mixed_cur_flows}, but also the limits of
$$\lim_{t \to \infty}\alpha(t)=\alpha_\infty \hbox{~~and~~} \lim_{t \to \infty}\beta(t)=\beta_\infty$$
exist. Therefore, we conclude that the metric $g_\infty=u_\infty^{\frac{4}{n-2}}g_0$ satisfies
$$R_{g_\infty}=\frac{\lambda_\infty}{\alpha_\infty}f\mbox{~~in~~}M\hspace{2mm}\mbox{~~and~~ }\hspace{2mm}h_{g_\infty}=\frac{\lambda_\infty}{\beta_\infty}h\mbox{~~on~~}\partial M.$$
This proves Theorem \ref{mainthm}.

Moreover, we give a criterion of the uniqueness of asymptotic limits for different initial data.

From \eqref{eqn:constrant_time_funcs} and \eqref{conserved_quantity}, we know that
\begin{equation}\label{eq:preserve_const}
a(a\alpha(t))^{\frac{n-2}{2}}+2(n-1)b(b\beta(t))^{n-2}=1
\end{equation}
holds for any time $t \geq 0$ and so does the asymptotic limit.

\begin{proposition}\label{prop:criterion_asy_limits}
Under the same assumptions of Theorem \ref{mainthm}, let $u_0^{(1)}, u_0^{(2)}$ be the qualified initial data and $u_0^{(1)}\neq u_0^{(2)}$, the corresponding asymptotic limits along the flows \eqref{mixed_cur_flows} with initial data $u_0^{(1)}, u_0^{(2)}$ are $u_\infty^{(1)}, u_\infty^{(2)}$ respectively. If $E[u_\infty^{(1)}]=E[u_\infty^{(2)}]$, then $u_\infty^{(1)}=u_\infty^{(2)}$.
\end{proposition}
\begin{proof}
For simplicity, we omit the subscript $\infty$ in the following proof, for instance, $u_1=u_\infty^{(1)}$ and $u_2=u_\infty^{(2)}$ and other involved quantities. Recall that $\lambda(t)=-E[u(t)]$ and assumption that $E[u_1]=E[u_2]$ , then $\lambda_1=\lambda_2=\lambda$. We have
\begin{align*}
R_{g_1}=\frac{\lambda}{\alpha_1}f, h_{g_1}=\frac{\lambda}{\beta_1}h,\\
R_{g_2}=\frac{\lambda}{\alpha_2}f, h_{g_2}=\frac{\lambda}{\beta_2}h
\end{align*}
and $\alpha_1, \beta_1$ satisfy \eqref{eq:preserve_const}, so do $\alpha_2,\beta_2$.

Suppose $v=u_2/u_1$, then $g_2=v^{\frac{4}{n-2}}g_1$, from the conformal invariant of $E[u]$, we have
\begin{align*}
\lambda&=-\int_{M}\left(\frac{4(n-1)}{n-2}|\nabla v|_{g_1}^2+\frac{\lambda}{\alpha_1} fv^2\right)d\mu_{g_1}+2(n-1)\int_{\pa M}\frac{\lambda}{\beta_1}(-h)v^2d\sigma_{g_1}\\
&\leq \frac{\lambda}{\alpha_1}\int_{M} (-f)v^2 d\mu_{g_1}+2(n-1)\frac{\lambda}{\beta_1}\int_{\pa M}(-h)v^2d\sigma_{g_1}\\
&\leq \frac{\lambda}{\alpha_1}\left(\int_M (-f)d\mu_{g_1}\right)^{\frac{2}{n}}\left(\int_M (-f)v^{\frac{2n}{n-2}}d\mu_{g_1}\right)^{\frac{n-2}{n}}\\
&\qquad+2(n-1)\frac{\lambda}{\beta_1}\left(\int_{\pa M}(-h)d\sigma_{g_1}\right)^{\frac{1}{n-1}}\left(\int_{\pa M}(-h)v^\frac{2(n-1)}{n-2}d\sigma_{g_1}\right)^{\frac{n-2}{n-1}}\\
&=\lambda [a(a\alpha_2)^{\frac{n-2}{2}}+2(n-1) b(b\beta_2)^{n-2}]=\lambda.
\end{align*}
Thus the equality holds if and only if $v=1$. This completes the proof.
\end{proof}

\section{Applications of the flow equations}\label{Sect7}

Without loss of generality, we assume the background metric $g_0$ enjoys the property that $R_{g_0}<0$ and $h_{g_0}<0$ thanks to \eqref{initial_metric}. Assume $Q(M,\pa M)$ is negative and finite. To emphasize the parameters $a,b>0$, we define a conformal invariant by
$$Y_{a,b}(M,\pa M):=\inf_{0\not\equiv \varphi \in H^1(M,g_0)}\q[\varphi]=\inf_{\varphi \in C_{a,b}}E[\varphi],$$
where
\begin{align*}
C_{a,b}=\Big\{&\varphi \in H^1(M,g_0);\\
&\quad a\Big(\int_M -f|\varphi|^{\frac{2n}{n-2}}d\mu_{g_0}\Big)^{\frac{n-2}{n}}+2(n-1)b \Big(\int_{\partial M}-h|\varphi|^{\frac{2(n-1)}{n-2}}d\sigma_{g_0}\Big)^{\frac{n-2}{n-1}}=1\Big\}.
\end{align*}
By the assumptions on $f$ and $h$, it is not hard to show that
\begin{enumerate}
\item[(a)] $Y_{a,b}(M,\pa M)>-\infty$.
\item[(b)]  $Y(M,\pa M)<0$ holds if and only if $Y_{a,b}(M,\pa M)<0$.
\item[(c)] If $\varphi \in C_{a,b}$ and $E[\varphi]\leq \Lambda$, then $\|\varphi\|_{H^1(M,g_0)}\leq C(\Lambda)$.
\end{enumerate}
\begin{proposition}\label{prop:cont_yab}
Fix any $a>0$, the function $b \mapsto Y_{a,b}(M,\pa M)$ is continuous in $(0,\infty)$. Fix any $b>0$, the same assertion also holds for $a \mapsto Y_{a,b}(M,\pa M)$ in $(0,\infty)$.
\end{proposition}
\begin{proof}
We only prove the first assertion. Let $b_m\to b$ as $m\to \infty$. For each $m$,  there exists $\phi_m\in C_{a,b_m}$ such that
\begin{equation}\label{est:Y_a,b_m}
E[\phi_m]\leq Y_{a,b_m}(M,\pa M)+\frac{1}{m}.
\end{equation}
For the above $\phi_m$, there exists a unique $\mu_m>0$ such that $\mu_m\phi_m\in C_{a,b}$. Then it follows from (b) and \eqref{est:Y_a,b_m} that $\|\phi_m\|_{H^1(M,g_0)}$ is uniformly bounded.  Observe that
\begin{align*}
1=\mu_m^2\left[a\Big(\int_M -f|\phi_m|^{\frac{2n}{n-2}}d\mu_{g_0}\Big)^{\frac{n-2}{n}}+2(n-1)(b_m-b+b) \Big(\int_{\partial M}-h|\phi_m|^{\frac{2(n-1)}{n-2}}d\sigma_{g_0}\Big)^{\frac{n-2}{n-1}}\right],
\end{align*}
it yields that $\mu_m\to 1$ as $m\to \infty$. Since
\begin{align*}
E[\mu_m\phi_m]\leq \mu_m^2Y_{a,b_m}(M,\pa M)+\frac{\mu_m^2}{m},
\end{align*}
when $m$ is large enough, we have
\begin{align*}
Y_{a,b}(M,\pa M)\leq Y_{a,b_m}(M,\pa M)+\frac{2}{m}.
\end{align*}
By a similar argument, we can obtain
$$Y_{a,b_m}(M,\pa M)\leq Y_{a,b}(M,\pa M)+\frac{2}{m}$$
for all sufficiently large $m$. Thus we obtain the continuity of $Y_{a,b}(M,\pa M)$ in $b$.
\end{proof}

Next let us prove existence and uniqueness of positive minimizers in $C_{a,b}$ for $\q[\varphi]$. Let $b_m\to b$  as $m \to \infty$. For any $m \in \mathbb{N}$ choose $\epsilon_m=\min\{\frac{1}{m}, \frac{-Y_{a,b_m}(M,\pa M)}{2}\}$, there exists $u^{(m)}_0 \in C_{a,b_m}$ such that
\begin{align}\label{minimizing:initial data}
Y_{a,b_m}(M,\pa M)\leq E[u^{(m)}_0]\leq Y_{a,b_m}(M,\pa M)+\epsilon_m<0.
\end{align}
Here we have used Proposition \ref{prop:cont_yab} and property (b). Consider the flow equations \eqref{mixed_cur_flows} with initial datum $u^{(m)}_0$, it is known from the proof of Theorem \ref{mainthm} that there exists a unique asymptotic limit $u^{(m)}_\infty \in C_{a,b_m}$ satisfying $Y_{a,b_m}(M,\pa M)\leq E[u^{(m)}_\infty]\leq E[u^{(m)}_0]$. This together with (c) imply that $||u^{(m)}_\infty||_{H^1(M,g_0)}$ is uniformly bounded. Going to a subsequence if necessary, we obtain
\[u^{(m)}_\infty\rightharpoonup u^*\quad \text{weakly in }H^1(M,g_0)\]
and
$$u^{(m)}_\infty\to u^*\quad \text{strongly in }L^2(M,g_0) \hbox{~~and~~} L^2(\pa M,g_0|_{\pa M}).$$
Consequently, it follows from Fatou's lemma, \eqref{minimizing:initial data} and Proposition \ref{prop:cont_yab} that
\begin{equation}\label{upper_est:energy minimizer}
E[u^*]\leq \liminf_{m\to \infty} E[u^{(m)}_\infty]\leq \lim_{m \to \infty} E[u_0^{(m)}]=Y_{a,b}(M,\pa M)<0.
\end{equation}
 Each $u^{(m)}_\infty$ solves

\begin{align}\label{problem:smc3}
\left\{\begin{array}{ll}
\displaystyle L_{g_0}u=\frac{\lambda(u)}{\alpha(u)} fu^{\frac{n+2}{n-2}},&\hspace{2mm}\mbox{ in }M,\\
\displaystyle B_{g_0}u=\frac{n-2}{2}\frac{\lambda(u)}{\beta(u)}h u^{\frac{n}{n-2}},&\hspace{2mm}\mbox{ on }\partial M,
\end{array}
\right.
\end{align}
where $\alpha(u), \beta(u),\lambda(u)$ are defined in \eqref{eqn:constrant_time_funcs_converge}.

\begin{lemma}\label{lem:lower bbd minimizing}
There exists a positive constant $C^\ast$ independent of $m$ such that
$$u^{(m)}_\infty(x) \geq C^\ast \quad \hbox{~~for all~~} x \in \bar M.$$
\end{lemma}
\begin{proof}
A standard maximum principle argument for the elliptic problem \eqref{problem:smc3} gives
\[u_\infty^{(m)}\geq \min\left\{\left( \frac{\alpha_\infty^{(m)}}{\lambda_\infty^{(m)}}\min_{\bar M}(f^{-1}R_{g_0})\right)^{\frac{n-2}{4}},\left(\frac{\beta_\infty^{(m)}}{\lambda_\infty^{(m)}}\min_{\pa M}(h^{-1}h_{g_0})\right)^{\frac{n-2}{2}}\right\}.\]
Observe that
\begin{align*}
\lambda_\infty^{(m)}=&-E[u_\infty^{(m)}],\\
\alpha_\infty^{(m)}=&\frac{1}{a}\left(\int_M -f(u_\infty^{(m)})^{\frac{2n}{n-2}}d\mu_{g_0}\right)^{\frac{2}{n}}\geq \frac{(-\max_{M}f)^{\frac{2}{n}}}{a}\left(\frac{E[u_\infty^{(m)}]}{Y(M,\pa M)}\right)^{\frac{2}{n-2}}\\
\geq&  \frac{(-\max_{M}f)^{\frac{2}{n}}}{a}\left(\frac{E[u_0^{(m)}]}{Y(M,\pa M)}\right)^{\frac{2}{n-2}}
\end{align*}
and similarly
\begin{align*}
\beta_\infty^{(m)}\geq \frac{(-\max_{\pa M}h)^{\frac{1}{n-1}}}{b_m}\left(\frac{E[u_0^{(m)}]}{2Q(M,\pa M)}\right)^{\frac{1}{n-2}}.
\end{align*}
Also from \eqref{minimizing:initial data} and Proposition \ref{prop:minimizer_continuity}, with a positive constant $C$ there holds $\frac{1}{C}\leq -E[u_0^{(m)}]\leq C$ for all $m \in \mathbb{N}$.
Putting the above estimates together, we obtain the desired assertion.
\end{proof}

Since $u^{(m)}_\infty\in C_{a,b_m}$ and $u_\infty^{(m)} \geq C^\ast$ by Lemma \ref{lem:lower bbd minimizing}, Fatou lemma shows
\begin{align*}
0<(\mu^\ast)^{-2}:=a\Big(\int_M -f{(u^*)}^{\frac{2n}{n-2}}d\mu_{g_0}\Big)^{\frac{n-2}{n}}+2(n-1)b \Big(\int_{\partial M}-h{(u^*)}^{\frac{2(n-1)}{n-2}}d\sigma_{g_0}\Big)^{\frac{n-2}{n-1}}\leq 1,
\end{align*}
which implies $\mu^*u^*\in C_{a,b}$. Then we have
\[Y_{a,b}(M,\pa M)\leq E[\mu^*u^*]=(\mu^*)^2E[u^*]\leq E[u^*]\leq Y_{a,b}(M,\pa M).\]
This forces $\mu^*=1$ and $E[u^*]=Y_{a,b}(M,\pa M)$, equivalently, $u^* \in C_{a,b}$ is a positive minimizer for $\q$. Moreover, noticing that
\[E[u^*]=Y_{a,b}(M,\pa M)=\lim_{m\to \infty} E[u^{(m)}_\infty]\]
by \eqref{upper_est:energy minimizer}, we obtain $||\nabla u^{(m)}_\infty||_{L^2(M,g_0)}\to ||\nabla u^*||_{L^2(M,g_0)}$  as $m \to \infty$ and thus  $u^{(m)}_\infty$ strongly converges to $u^*$  in $H^1(M,g_0)$.  In addition, the regularity theory of elliptic equations (cf. \cite{cherrier}) shows that $u^\ast$ is smooth. The uniqueness of such minimizers can follow from Proposition \ref{prop:criterion_asy_limits}.

\begin{remark}\label{rem:subcritical_approx}
Under the same assumptions of Theorem \ref{mainthm}, we sketch another subcritical approximation method to derive minimizer for $\q$. For $1\leq q \leq \frac{n+2}{n+2}$, let
\begin{align*}
\mu_q=&\inf_{0\not \equiv u \in H^1(M,g_0)}\frac{E[u]}{a\left(\int_M -f |u|^{q+1} d\mu_{g_0}\right)^{\frac{2}{q+1}}+2(n-1)b \left(\int_{\pa M}-h|u|^{\frac{q+3}{2}}d\sigma_{g_0}\right)^{\frac{4}{q+3}}}\\
:=& \inf_{0\not \equiv u\in H^1(M,g_0)}F_q[u].
\end{align*}
Using minimizing sequence argument, we can show that for each $1\leq q<\frac{n+2}{n-2}$, there exists a positive smooth minimizer $u_q$ for $\mu_q$, which satisfies
$$a\left(\int_M -f u_q^{q+1} d\mu_{g_0}\right)^{\frac{2}{q+1}}+2(n-1)b \left(\int_{\pa M}-hu_q^{\frac{q+3}{2}}d\sigma_{g_0}\right)^{\frac{4}{q+3}}=1$$
and solves
\begin{equation*}
\begin{cases}
\displaystyle L_{g_0}u_q=\frac{\lambda_q}{\alpha_q}f u_q^q &\qquad\hbox{~~in~~} M,\\
\displaystyle B_{g_0}u_q=\frac{n-2}{2}\frac{\lambda_q}{\beta_q}h u_q^{\frac{q+1}{2}} &\qquad\hbox{~~on~~} \pa M,
\end{cases}
\end{equation*}
where
$$\lambda_q=-F_q[u_q],~~\alpha_q=\frac{1}{a}\left(\int_M -f u_q^{q+1}d\mu_{g_0}\right)^{\frac{q-1}{q+1}},~~\beta_q=\frac{1}{b}\left(\int_{\pa M}-hu_q^{\frac{q+3}{2}}d\sigma_{g_0}\right)^{\frac{q-1}{q+3}}.$$
Then it is not hard to get $\limsup_{q\nearrow \frac{n+2}{n-2}}\mu_q \leq \mu_{\frac{n+2}{n-2}}=Y_{a,b}(M,\pa M)$. Let $\underline \mu:=\liminf_{q \nearrow \frac{n+2}{n-2}}\mu_{q}$. Choose a sequence $q_i \nearrow \frac{n+2}{n-2}$ as $i \to \infty$ and $q_0 \leq q_i < \frac{n+2}{n-2}$ for some $q_0 \in (1,\frac{n+2}{n-2})$, such that $\lim_{i\to \infty}\mu_{q_i}=\underline \mu$. Let $u_{q_i}$ be the positive minimizer for $F_{q_i}$. By property of $u_{q_i}$ and (a), we can prove that $u_{q_i}$ is uniformly bounded in $H^1(M,g_0)$ and uniformly bounded away from zero by a similar argument in Lemma \ref{lem:lower bbd minimizing}. Up to a subsequence, we obtain $u_{q_i} \rightharpoonup \Psi$ in $H^1(M,g_0)$ as $i\to \infty$. Moreover, we proceed as the above argument after Lemma \ref{lem:lower bbd minimizing} to show that $\Psi$ is a minimizer for $F_{\frac{n+2}{n-2}}=\q$ and $\underline \mu\geq Y_{a,b}(M,\pa M)$. Thus we also have $\lim_{q \nearrow \frac{n+2}{n-2}}\mu_{q}=Y_{a,b}(M,\pa M)$.
\end{remark}

Hence we now in a position to state:

\begin{proposition}\label{prop:minimizer_continuity}
For any $a,b>0$, $Y_{a,b}(M,\pa M)$ is achieved by a unique positive minimizer $u_{a,b} \in C_{a,b}$. Moreover, the functions $Y_{a,b}(M,\pa M), \alpha_{a,b}=\alpha(u_{a,b})$ and $\beta_{a,b}=\beta(u_{a,b})$ are continuous in $a$ for each fixed $b$ and in $b$ for each fixed $a$.
\end{proposition}
\begin{proof}
We only need to prove the second assertion. Fix $a>0$, let $b_m \to b$ as $m \to \infty$. For each $m$, let $u_{a,b_m}\in C_{a,b_m}$ be a minimizer satisfying problem \eqref{problem:smc3} and $E[u_{a,b_m}]=Y_{a,b_m}(M,\pa M)$. With the same argument as above, up to a subsequence we conclude $u_{a,b_m}$ strongly converges in $H^1(M,g_0)$ to a minimizer $u_{a,b} \in C_{a,b}$ for $\q$ as $m \to \infty$. From Sobolev embedding and the uniqueness of $u_{a,b}$, we conclude that $\alpha_{a,b},\beta_{a,b}$ are continuous in $b$. For any fixed $b>0$, the continuity of $\alpha_{a,b}, \beta_{a,b}$ on $a$ can similarly obtained.
\end{proof}

A natural question naturally arises after Theorem \ref{mainthm}: Do there exist any $a,b>0$ and a metric $g_{a,b} \in [g_0]$ related to $a,b$ in some way such that $R_{g_{a,b}}=f$ and $h_{g_{a,b}}=h$? In order to answer this question, we define $g_{a,b}=u^{\frac{4}{n-2}}_{a,b}g_0$, where $u_{a,b}\in C_{a,b}$ is the unique minimizer for $\q$ in Proposition \ref{prop:minimizer_continuity}. Up to some positive constant in metric $g_{a,b}$, we may assume
\[R_{g_{a,b}}=f,\quad h_{g_{a,b}}=\frac{(-Y_{a,b}(M,\pa M)\alpha_{a,b})^{\frac{1}{2}}}{\beta_{a,b}}h.\]
We first need the following:
\begin{lemma}
(1) There exists some constant $C=C(f,h)>0$ such that
\begin{align}\label{ieq:yab_lowerbound}
 \max\left\{\frac{Y(M,\pa M)}{a},\frac{Q(M,\pa M)}{b}\right\}\leq CY_{a,b}(M,\pa M).
\end{align}

There exists some constant $C=C(f,h)>0$ such that if $a\geq C(f,g)b$, then
\begin{align}\label{ieq:a>b}
Y_{a,b}(M,\pa M)\leq \frac{1}{Ca}Y(M,\pa M).
\end{align}
If $b\geq C(f,h)a$, then
\begin{align}\label{ieq:b>a}
Y_{a,b}(M,\pa M)\leq \frac{1}{Cb}Q(M,\pa M).
\end{align}

(2) $\exists\, a,b>0$ such that ${(-Y_{a,b}(M,\pa M)\alpha_{a,b})^{\frac 12}}={\beta_{a,b}}$.
\end{lemma}
\begin{proof}
(1) Estimate \eqref{ieq:yab_lowerbound} follows from the definitions of conformal invariants and assumptions of $f,h$.

To show \eqref{ieq:a>b}, choose some $\varphi \in H^1(M,g_0)\setminus \{0\}$ such that
\[E[\varphi]\leq \frac{1}{2}Y(M,\pa M)\left(\int_M |\varphi|^{\frac{2n}{n-2}}d\mu_{g_0}\right)^{\frac{n-2}{n}}.\]
Let $\mu$ be the unique positive constant such that $\mu \varphi\in C_{a,b}$. If the constant $C=C(f,h)$ is chosen as
\[a \geq \frac{\left(\int_{\pa M}-h|\varphi|^{\frac{2(n-1)}{n-2}}d\mu_{g_0}\right)^{\frac{n-2}{n-1}}}{\left(\int_M -f|\varphi|^{\frac{2n}{n-2}}d\mu_{g_0}\right)^{\frac{n-2}{n}}}b :=C(f,h) b,\]
and by $\mu \varphi\in C_{a,b}$, we have
\[\mu^2a\left(\int_M -f|\varphi|^{\frac{2n}{n-2}}d\mu_{g_0}\right)^{\frac{n-2}{n}}\geq \frac 12.\]
Thus we obtain
\[Y_{a,b}(M,\pa M)\leq E[\mu \varphi]=\mu^2E[\varphi]\leq \frac{1}{2a}\frac{E[\varphi]}{\left(\int_M -f|\varphi|^{\frac{2n}{n-2}}d\mu_{g_0}\right)^{\frac{n-2}{n}}}\leq \frac{1}{Ca}Y(M,\pa M).\]

Similarly we can prove estimate \eqref{ieq:b>a}.

(2) Notice that
\begin{align}\label{qty:continuity}
&{(-Y_{a,b}(M,\pa M)\alpha_{a,b})^{\frac 12}}{\beta_{a,b}}^{-1}\no\\
=&(-Y_{a,b}(M,\pa M))^{\frac{1}{2}}\frac{b}{a^{\frac{1}{2}}}\Big(\int_M -fd\mu_{g_{a,b}}\Big)^{\frac{1}{n}}\Big(\int_{\partial M}-hd\sigma_{g_{a,b}}\Big)^{-\frac{1}{n-1}}.
\end{align}
Fix $b>0$ first. When $a$ small enough, it follows from \eqref{ieq:b>a} that
\[Y(M,\pa M)\leq \frac{E[u_{a,b}]}{\left(\int_{M}d\mu_{g_{a,b}}\right)^{\frac{n-2}{n}}}=\frac{Y_{a,b}(M,\pa M)}{\left(\int_{M}d\mu_{g_{a,b}}\right)^{\frac{n-2}{n}}}\leq \frac{Q(M,\pa M)}{Cb\left(\int_{M}d\mu_{g_{a,b}}\right)^{\frac{n-2}{n}}},\]
which means
$$\int_{M}-f d\mu_{g_{a,b}}\geq \left(\frac{Q(M,\pa M)}{C bY(M,\pa M)}\right)^{\frac{n}{n-2}}.$$
Also $u_{a,b}\in C_{a,b}$ implies that
$$\int_{\pa M}-h d\mu_{g_{a,b}}\leq \left(\frac{1}{2(n-1)b}\right)^{\frac{n-1}{n-2}}.$$
Together with \eqref{ieq:b>a} and \eqref{qty:continuity}, we conclude that for any fixed $b$, there holds
\begin{align}\label{limit_a}
\lim_{a \to 0}(-Y_{a,b}(M,\pa M)\alpha_{a,b})^{\frac 12}{\beta_{a,b}}^{-1}= \infty.
\end{align}

Next fix $a>0$. When $b$ is small enough, it follows from \eqref{ieq:a>b} that
\[Q(M,\pa M)\leq \frac{E[u_{a,b}]}{\left(\int_{\pa M}d\sigma_{g_{a,b}}\right)^{\frac{n-2}{n-1}}}=\frac{Y_{a,b}(M,\pa M)}{\left(\int_{\pa M}d\sigma_{g_{a,b}}\right)^{\frac{n-2}{n-1}}}\leq \frac{Y(M,\pa M)}{Ca\left(\int_{\pa M}d\sigma_{g_{a,b}}\right)^{\frac{n-2}{n-1}}}.\]
This means
$$\int_{\pa M}-hd\sigma_{g_{a,b}}\geq \left(\frac{Y(M,\pa M)}{Ca Q(M,\pa M)}\right)^{\frac{n-1}{n-2}}.$$
Also $u_{a,b}\in C_{a,b}$ implies
$$\int_M -f d\mu_{g_{a,b}}\leq a^{\frac{n}{2-n}}.$$
From \eqref{ieq:yab_lowerbound}, we have
$$-Y_{a,b}(M,\pa M)\leq- \frac{Y(M,\pa M)}{a}.$$
Combining these estimates and \eqref{qty:continuity}, we obtain for any fixed $a$, there holds
\begin{align}\label{limit_b}
\lim_{b \to 0}{(-Y_{a,b}(M,\pa M)\alpha_{a,b})^{\frac 12}}{\beta_{a,b}}^{-1}= 0.
\end{align}
Consequently, by Proposition \ref{prop:minimizer_continuity}, \eqref{limit_a} and \eqref{limit_b},  an elementary continuity argument shows that there exist some $a,b>0$ satisfying
$${(-Y_{a,b}(M,\pa M)\alpha_{a,b})^{\frac 12}}{\beta_{a,b}}^{-1}=1$$
as required.
\end{proof}

Therefore we can give a positive answer to that question and obtain the following result.

\begin{theorem}
Under the same assumptions in Theorem \ref{mainthm}, there exist some $a,b>0$ and the corresponding minimizer $u_{a,b} \in C_{a,b}$ for $\q$ such that a positive constant multiple of the metric $u_{a,b}^{4/(n-2)}g_0$ has scalar curvature $f$ and mean curvature $h$.
\end{theorem}

\end{document}